\documentclass[reqno, 10pt]{amsart}
\usepackage{amssymb,amsmath,amsfonts,bm, mathtools,dsfont}
\usepackage{epsfig}
\usepackage{graphicx}
\usepackage{esint}
\usepackage{enumerate}
\usepackage{verbatim}
\usepackage{xcolor}
\usepackage{bbm}
\usepackage{caption}
\usepackage{geometry}
\usepackage{enumitem}
\usepackage{bm}
\usepackage{lipsum}
\usepackage{mathtools}
\usepackage{esint}
\usepackage[makeroom]{cancel}

\newcommand{\Z}{\mathbb{Z}}

\newcommand{\R}{\mathbb{R}}
\newcommand{\N}{\mathbb{N}}

\newcommand{\1}{\mathbbm{1}}

\newcommand{\T}{\mathbb{T}}
\newcommand{\C}{\mathbb{C}}
\newcommand{\ve}{\varepsilon}

\date{} 

\theoremstyle{plain}
\newtheorem{theorem}{Theorem}[section]
\newtheorem{definition}[theorem]{Definition}
\newtheorem{proposition}[theorem]{Proposition}

\newtheorem{lemma}[theorem]{Lemma}
\newtheorem{corollary}[theorem]{Corollary}
\newtheorem{remark}[theorem]{Remark}

\numberwithin{theorem}{section}
\numberwithin{equation}{section}
\numberwithin{figure}{section}

\let\oldtocsection=\tocsection
 
\let\oldtocsubsection=\tocsubsection
 
\let\oldtocsubsubsection=\tocsubsubsection
 
\renewcommand{\tocsection}[2]{\hspace{0em}\oldtocsection{#1}{#2}}
\renewcommand{\tocsubsection}[2]{\hspace{1em}\oldtocsubsection{#1}{#2}}
\renewcommand{\tocsubsubsection}[2]{\hspace{2em}\oldtocsubsubsection{#1}{#2}}

\begin{document}

\parskip=4pt

%\vspace*{1cm}
\title[]
{Complex-Valued Modified Zakharov Kuznetsov equation}

\author[Carlos E. Kenig]{Carlos E. Kenig}
\address{Carlos E. Kenig, 
Department of Mathematics,
University of Chicago}
\email{cek@math.uchicago.edu}

\author[Nata\v{s}a Pavlovi\'{c}]{Nata\v{s}a Pavlovi\'{c}}
\address{Nata\v{s}a Pavlovi\'{c},  
Department of Mathematics, The University of Texas at Austin.}
\email{natasa@math.utexas.edu}

\author[Gigliola Staffilani]{Gigliola Staffilani}
\address{Gigliola Staffilani,
Department of Mathematics,
Massachusetts Institute of Technology}
\email{gigliola@math.mit.edu}

\author[Luisa Velasco]{Luisa Velasco}
\address{Luisa Velasco,  
Department of Mathematics, The University of Texas at Austin.}
\email{lmvelasco@utexas.edu}
\begin{abstract} 
Motivated by the introduction of the Zakharov-Kuznetsov equation as a higher dimensional generalization of the Korteweg-de Vries equation, in this paper we introduce the modified Zakharov-Kuznetsov (mZK) equation as a 2-dimensional generalization of the complex-valued modified Korteweg-de Vries equation. We initiate the mathematical analysis of the mZK equation
on $\T^2$ by proving a local well-posedness result in Sobolev spaces and by establishing a failure of uniform continuity. We also note that the real-valued version of the mZK equation has physical significance. 

\end{abstract}
\maketitle
\tableofcontents

\section{Introduction}\label{sec:intro}

In this work we will study the local well-posedness of the equation 
\begin{align}\label{eq:mZK}
\partial_t u + \partial_{x_1}\Delta u = 6|u|^2\partial_{x_1}u \qquad (t,(x_1,x_2)) \in \R \times \T^2
\end{align}
where $u: \R\times \T^2 \to \C$, which we call the \textit{modified} Zakharov-Kuznetsov (mZK) equation. We introduce this equation as a 2-dimensional generalization of the complex-valued modified Korteweg-de Vries (mKdV) equation
\begin{align}\label{eq:mKdV}
    \partial_t u + \partial_{x}^3 u = |u|^2 \partial_{x} u, \qquad (t,x) \in \R \times \R, \T,
\end{align} studied in e.g. \cite{Birnir1996,Chapouto2021,Chapouto2023,Grunrock2004, Kenig2025}. Before we describe the details of this generalization, we recall what is known about higher dimensional generalizations of the 1-dimensional Korteweg-de Vries (KdV) equation.

In 1 dimension, the real-valued KdV equation 
\begin{align}\label{eq:KdV}
    \partial_tu + \partial_x^3u = \partial_{x}(u^2), \qquad (t,x) \in \R\times \R, \T
\end{align}
is a well-known and well-studied model, see e.g. \cite{KdV1895, Bourgain1993KdV, Kenig1991, ChenGuo2020, Colliander2003, Guo2009, Kato1983}, for many systems of weakly nonlinear interacting waves including shallow water waves \cite{Brooke1967,KdV1895} and ion-acoustic waves in various plasma environments \cite{Taniuti1966}. We also note that the equation \eqref{eq:KdV} is known to be an integrable system \cite{Gardner1974,Lax1968}. The KdV equation \eqref{eq:KdV} has been generalized to higher dimensions. In particular, a 2 dimensional generalization of the KdV equation is given by the Zakharov-Kuznetsov (ZK) equation
\begin{align}\label{eq:ZK}
    (\partial_t + \partial_{x_1}\Delta) u = \partial_{x_1}(u^2),\qquad (t,(x_1,x_2)) \in \R\times \R^2,\, \T^2,\, \R \times \T
\end{align}
which models the effects of a magnetic field on ion-acoustic sound waves in plasma \cite{GrunrockHerr2014, Linares2019, MolinetPilod2015, MunroParkes1999, Seadawy2014}. The anisotropy in equation \eqref{eq:ZK} in the $x_1$ direction reflects that the plasma waves propagate more freely along magnetic field lines. In particular, the magnetic field $\mathbf{B} = B_0\mathbf{e}_{x_1}$ is assumed to be uniform and aligned along the $x_1$-direction.
We also note that this type of ZK equation with a weak nonlinearity is a starting point for the derivation of the three wave kinetic equation \cite{Staffilani2021,HaRoStTr2022}. 

Motivated by the introduction of the ZK equation \eqref{eq:ZK} as a generalization of the KdV equation \eqref{eq:KdV}, we think of the mZK equation \eqref{eq:mZK} as a 2-dimensional generalization of the complex-valued mKdV equation \eqref{eq:mKdV}.

In addition to this mathematical motivation, the real-valued version of \eqref{eq:mZK} has physical significance. In particular, Chiron showed \cite{Chiron2014} that the cubic nonlinearity appears in the limit of degenerate cases of the Euler-Poisson system. 
Moreover, a version of the equation \eqref{eq:mZK} has been used as a physical model for ion-acoustic waves in plasma.  More precisely, the work of Younas, Ren, Baber, Yasin, Shahzah \cite{Younas2023} studies the (3 + 1)-dimensional equation given below
    \begin{align}\label{eq:gKV-ZK}
        \partial_t u + \beta \partial_{x_1}^3u + \nu\partial_{x_1}(\Delta u) = \alpha u^2\partial_{x_1}u,
    \end{align}
which corresponds to \eqref{eq:mZK} for a specific choice of coefficients $\alpha,\beta, \nu$. 
The equation \eqref{eq:gKV-ZK} serves as a model for the effects of a magnetic field on weak ion-acoustic waves. In \cite{Younas2023}, the authors employ an expansion method to obtain an expression for soliton solutions and computational methods are applied to draw the profiles of these solutions and study their dynamics.

In this paper, we initiate the mathematical analysis of the equation \eqref{eq:mZK}. The main results are local well-posedness of solutions to \eqref{eq:mZK} and failure of uniform continuity. We give details about each of these results in the subsequent sections.

    Since we consider complex-valued solutions, the nonlinearity cannot be expressed as a total derivative, in contrast to the real-valued versions mKdV \eqref{eq:mKdV}, KdV \eqref{eq:KdV} and ZK \eqref{eq:ZK} equations as well as their real-valued generalizations with $p$-order nonlinearities (e.g. equation \eqref{eq:gKV-ZK}). This structural difference has significant consequences for both the well-posedness theory and the failure of uniform continuity arguments developed in this work.

\subsection{Local Theory}

Recently, Kenig, Nahmod, Pavlović, Staffilani, and Visciglia \cite{Kenig2025}, studied the well-posedness of the complex-valued mKdV \eqref{eq:mKdV}. 
We adapt the program implemented in \cite{Kenig2025} (which originated in the works of Ionescu and Kenig \cite{IoKe2007} and Koch and Tzvetkov \cite{Koch2003}) to recover the derivative in the nonlinearity in order to prove local well-posedness for the higher dimensional versions of this equation \eqref{eq:mZK}. Since the mKdV \eqref{eq:mKdV} and the mZK \eqref{eq:mZK} equations share a common nonlinearity, the difference in the local theory stems from the mixing of derivatives in the linear part of the mZK equation \eqref{eq:mZK}.

In order to state our result, we first define our notion of solution. 
\begin{definition}\label{def:strong solution}
    We say that $u \in C([-T,T], H^s)$, $s > 0$ is a \textit{strong solution} of \eqref{eq:mZK} if 
    \begin{align}
        u(t) = W(t)u_0 + \int W(t- \tau)|u|^2\partial_{x_1}u(\tau)d\tau.
    \end{align}
    where \begin{align}\label{eq:W}W(t)\phi = \int_{\R^d}\sum\limits_{\mathbf{k} \in \Z^2} e^{i(\mathbf{k}\cdot \mathbf{x} + ( k_1^3 + k_1k_2^2)t)}\hat{\phi}(\mathbf{k}).\end{align}
\end{definition}

Now, we can state the main local well-posedness result.

\begin{theorem}[Local well-posedness]\label{thm:lwp}
    Assume $u_0 \in H^s(\T^2)$ for $s > \frac{5}{3}$. Then there exists a $T > 0$ and a unique strong solution $u \in C([-T,T],H^s)$ to the initial value problem for \eqref{eq:mZK} corresponding to $u_0$.
\end{theorem}

The proof of this theorem is the content of Section \ref{sec:lwp}. Since $\frac{5}{3} < 2$, the standard energy method does not apply. Thus, we proceed with the recovery of the derivative in the nonlinearity in the spirit of \cite{Kenig2025} as follows.

\begin{itemize}
    \item For each $N \in \N$, we introduce a system truncated in frequency \eqref{eq:fin dim} below. The well-posedness of this system follows easily from the theory of nonlinear ODE. However, these solutions may depend critically on the truncation parameter $N$. Thus, we need to work to obtain uniform bounds in $N$ in order to pass to the limit.
    \item In order to obtain these uniform bounds, we crucially apply the Strichartz estimate which originates in the work \cite{Linares2019} of Linares, Panthee, Robert, and Tzvetkov on the ZK equation  \eqref{eq:ZK} and is formulated for our purposes in Lemma \ref{lem:strichartz}. The proof of this Strichartz estimate relies on that fact one can diagonalize the dispersion relation $k_1^3 + k_1k_2^2$ by rotating coordinates. While the mZK \eqref{eq:mZK} and ZK \eqref{eq:ZK} equations share a common linear part, we must still deal with the cubic nonlinearity.
    \item Next, we show that the sequence of solutions to the truncated problem is Cauchy and that a solution $u$ of the mZK \eqref{eq:mZK} is obtained as a limit of the truncated system. 
    \item We first demonstrate that this solution $u$ is weakly continuous in $H^s$ and then prove that $u$ is continuous in the strong $H^s$ topology.
    \item Lastly, we prove uniqueness of solutions. 
\end{itemize}

In our proof of existence of solutions we use the Kato-Ponce estimate on $\T^d$ expressed in terms of the multiplier $J^s$ \eqref{eq:Js on torus} for $s \geq 1$, (see Lemma \ref{lem:KP}). Since we did not find the version of this estimate on $\T^d$ with $d > 1$ in the literature, in order to make the paper self-contained, in Section Appendix \ref{sec:apB} we state and prove this Kato-Ponce estimate. The proof is inspired by the proof of the Kato-Ponce estimate on $\T$ obtained by Kenig and Ionescu \cite{IoKe2007} and a transference principle formulated and proved by Roncal and Stinga \cite{RoncalStinga} for the fractional laplacian.

Finally, we note that our result can be understood as an analogue of the well-posedness result obtained by Linares, Panthee, Robert, and Tzvetkov \cite{Linares2019} for the ZK equation \eqref{eq:ZK}.

Based on the results that are obtained for the real and complex valued KdV, we expect that the real valued version of the equation \eqref{eq:mZK} can be addressed using similar arguments. In fact, we expect potentially better behavior, including well-posedness at lower regularity via $X_{s,b}$-type spaces.

\subsection{Failure of Uniform Continuity for the mZK equation}
Lastly, inspired by the work of Linares, Panthee, Robert, and Tzvetkov \cite{Linares2019} on the ZK equation and the work of Herr \cite{Herr2006} on the derivative nonlinear Schr\"odinger equation, we show the failure of uniform continuity of solutions of \eqref{eq:mZK}. This is the content of Section \ref{sec:fail_unif_cont}.

\begin{theorem}\label{thm:failure}
    Let $s > 5/3$. There exist two positive constants $c, C$ and two sequences $\{u_m\}$ and $\{v_m\}$ solutions to \eqref{eq:mZK} in $C([0,1]: H^s(\T^2))$ such that 
    \begin{align}
        \sup_{t \in [0,1]}\|u_m(t)\|_{H^s} + \|v_m(t)\|_{H^s} \leq C
    \end{align}
    and satisfy 
    \begin{align}\label{eq:fail init data}
        \lim\limits_{m\to \infty}\|u_m(0) - v_m(0)\|_{H^s} = 0,
    \end{align}
    but for every $t \in [0,1]$, 
    \begin{align}\label{eq:fail conclusion}
        \lim\inf_{m\to \infty}\|u_m(t) - v_m(t)\|_{H^s} \geq ct.
    \end{align}
\end{theorem}

We will prove the failure of uniform continuity of the flow map for \eqref{eq:mZK} by adapting the procedure outlined in \cite{Linares2019} and \cite{Herr2006} to the context of the complex-valued mZK equation \eqref{eq:mZK}. 
We will carefully construct a family of approximate solutions from which we will demonstrate growth in the sense of \eqref{eq:fail conclusion}.  
We note that while we draw inspiration from \cite{Linares2019}, our construction of approximate solutions is very different to accommodate the cubic nonlinearity and complex-valued solutions.  In particular, the higher order nonlinearity requires more careful treatment in estimating the difference between the approximate and true solutions. Further, we find the example of \cite{Herr2006} for the Derivative NLS more suitable for constructing the families of approximate solution for \eqref{eq:mZK}: see also \cite{ChristCollianderTao} and references therein.

\textbf{Acknowledgements.} CK is supported in part by NSF DMS-2153794 and DMS-2052710.  NP is partially supported by the NSF grants DMS-1840314, DMS-2052789 and DMS-2511517. GS is partially supported by the NSF grant DMS-2306378 and the Simons
Foundation Collaboration Grant on Wave Turbulence.
LV is partially supported by the NSF grants DMS-1840314, DMS-2052789 and DMS2511517 through NP’s grants and the NSF GRFP. 
This material is based upon work supported by the
National Science Foundation under Grant No. DMW-2424139, while NP, GS and LV were in residence at
the Simons Laufer Mathematical Sciences Institute in Berkeley, California during the Fall 2025 semester.
The authors would like to thank the Simons Laufer Mathematical Institute for their kind hospitality.

\section{Local Theory}\label{sec:lwp}

In this section, we will prove Theorem \ref{thm:lwp}. First, we introduce and address the finite dimensional problem obtained by truncating in frequency. That is, for $N \in \N$, we consider the following Cauchy problem:
\begin{align}\label{eq:fin dim}
    \begin{cases}
        (\partial_t +\partial_{x_1}\Delta) u = 6\Pi_N(|\Pi_N u|^2\partial_{x_1} \Pi_Nu), \quad (t,\mathbf{x}) \in \R \times \T^2,\\
        u(x,0) = u_0.
    \end{cases}
\end{align}
where $\Pi_N$ denotes the projection 
\begin{align*}
    \Pi_N\left(\sum\limits_{n \in \Z^2}u_ne^{in\cdot x}\right) = \sum\limits_{|n| \leq N}u_ne^{in\cdot x}.
\end{align*}

We can also define 
\begin{align*}
    \Pi_{> N} = I - \Pi_N
\end{align*}
so that 
\begin{align*}
    \Pi_{> N}\left(\sum\limits_{n \in \Z^2}u_ne^{in\cdot x}\right) = \sum\limits_{|n| > N}u_ne^{in\cdot x}.
\end{align*}

\begin{remark}\label{rem:ODE wellposed}
    For each $N,$ the Cauchy problem \eqref{eq:fin dim} is well-posed in $H^s$ for $s \geq 0$. This can be seen if we recognize that the solution $u_N$ can be expressed as 
    \begin{align}
        u_N = \Pi_Nu_N + \Pi_{> N}u_N,
    \end{align}
    where the second term evolves according to the linear mZK flow and the first term evolves according to a system of $\sim N^2$ nonlinear ODEs. However, we note that the solutions from this theory may depend critically on $N$. For the rest of this section, we define $\Phi_N(t)$ to be flow associated to \eqref{eq:fin dim}. 

\end{remark}

We will proceed by proving estimates that are \textit{independent of $N$} and then pass to the limit.

We adapt the notation from \cite{Kenig2025}:
\begin{align}\label{eq:Js on torus}
    \widehat{J^sg}(\mathbf{k})& = (1 + |\mathbf{k}|^2)^{s/2}\hat{g}(\mathbf{k})
\end{align}

 Next we show a local in time energy estimate that is uniform in $N$ for the flow $\Phi_N$ corresponding to \eqref{eq:fin dim}. 
\begin{lemma}\label{lem:energy bound}
Let $s > 1$. Then there exists $C > 0$ such that 
\begin{align}\label{eq:lem3.3}
    \frac{d}{dt}\|J^s u_N\|_{L^2}^2 = \frac{d}{dt}\|J^s \Pi_N u_N\|_{L^2}^2 \leq C\left(\|J^s \Pi_Nu_N\|_{L^2}^3\|\nabla \Pi_Nu_N\|_{L^\infty} + \|J^s\Pi_Nu_N\|_{L^2}^4\right)
\end{align}
where $u_N(t) = \Phi_N(t)u_0$ where $\Phi_N$ is the flow corresponding to \eqref{eq:fin dim}.
\end{lemma}
\begin{proof}

Before we proceed, for $\boldsymbol{\ell} \in \Z^2$, we define the $\boldsymbol{\ell}$th Fourier coefficient of $u_N$ as follows
\begin{align}
    u_{N,\boldsymbol{\ell}} = \frac{1}{(2\pi)^2}\int_{\T^2}u_N(\mathbf{x})e^{-i\boldsymbol{\ell}\cdot \mathbf{x}}d\mathbf{x}. 
\end{align}

We make use of the following decomposition,
\begin{align}
    u_N = \Pi_N u_N + \Pi_{> N}u_N, 
\end{align}
to write
\begin{align}
    \|J^su_N\|_{L^2}^2 &= \int_{\T^2}J^s\left(\Pi_N u_N + \Pi_{>N}u_N\right)\overline{J^s\left(\Pi_N u_N + \Pi_{> N}u_N\right)} \,dz\notag\\
     & =\int_{\T^2}J^s\Pi_N u_N \overline{J^s\Pi_Nu_N}\,dz + 2Re \int_{\T^2}J^s\Pi_N u_N \overline{J^s\Pi_{>N}u_N}\,dz + \int_{\T^2}J^s\Pi_{>N}u_N\overline{J^s\Pi_{> N}u_N}\,dz\notag\\
         & =  \|J^s \Pi_N u_N\|_{L^2}^2 + \|J^s\Pi_{>N}u_N\|_{L^2}^2,
\end{align}
where we obtain the last line thanks to\begin{align}\label{eq:term that is zero}
\int_{\T^2}J^s\Pi_Nu_N\overline{J^s\Pi_{>N}u_N}\,dz = 0,
\end{align} since $\Pi_Nu_N$ and $\Pi_{>N}u_N$ have disjoint frequency support.

Taking the derivative in time of $\|J^s\Pi_{>N}u_N\|_{L^2}$, we observe that
\begin{align}
    \frac{d}{dt}\|J^s\Pi_{>N}u_N\|_{L^2}^2 &= 2Re\Bigg(\sum\limits_{|\boldsymbol{\ell}| > N}(1 + |\boldsymbol{\ell}|^2)^s\dot{u}_{N,\boldsymbol{\ell}}\overline{u}_{N,\boldsymbol{\ell}}\Bigg)\notag\\
         & = 2Re \Bigg(\sum\limits_{|\boldsymbol{\ell}| > N}(1 + |\boldsymbol{\ell}|^2)^s\Bigg[i(\ell_{1}^3 + \ell_{1}\ell_{2}^2)u_{N,\ell}\Bigg]\overline{u}_{N,\ell}\Bigg)\notag\\
         & = 2\sum\limits_{|\boldsymbol{\ell}| > N}(1 + |\boldsymbol{\ell}|^2)^sRe\Big[i(\ell_{1}^3 + \ell_{1}\ell_{2}^2)|u_{N,\boldsymbol{\ell}}|^2\Big]\notag\\
         & = 0.\notag
\end{align}

Thus, 
\begin{align}\label{eq:full u_N}
    \frac{d}{dt}\|J^su_N\|_{L^2}^2 = \frac{d}{dt}\|J^s \Pi_N u_N\|_{L^2}^2 + \frac{d}{dt}\|J^s \Pi_{>N}u_N\|_{L^2}^2 = \frac{d}{dt}\|J^s\Pi_N u_N\|_{L^2}^2.
\end{align}
Then, taking the derivative in time of $\|J^s \Pi_Nu_N\|_{L^2}^2$ and substituting from \eqref{eq:fin dim}, we have
    \begin{align}\label{eq:truncated t derivative}
        \frac{d}{dt}\|J^s \Pi_N u_N\|_{L^2}^2 &= 2Re\Bigg(\sum\limits_{|\boldsymbol{\ell}| \leq N}(1 + |\boldsymbol{\ell}|^2)^s\dot{u}_{N,\boldsymbol{\ell}}\overline{u}_{N,\boldsymbol{\ell}}\Bigg)\notag\\
         & = 2Re \Bigg(\sum\limits_{|\boldsymbol{\ell}| \leq N}(1 + |\boldsymbol{\ell}|^2)^s\Bigg[i(\ell_{1}^3 + \ell_{1}\ell_{2}^2)u_{N,\ell}\Bigg]\overline{u}_{N,\ell} \notag\\
         & \qquad \qquad + i6\sum\limits_{\boldsymbol{\ell} \leq N}(1 + |\boldsymbol{\ell}|^2)^s\Bigg[\sum\limits_{\substack{\boldsymbol{\ell} = \mathbf{k}_1 - \mathbf{k}_2 + \mathbf{k}_3\\\mathbf{k}_1,\mathbf{k}_2,\mathbf{k}_3 \leq N}}k_{3,1}u_{N,\mathbf{k}_1}\overline{u}_{N,\mathbf{k}_2}u_{N,\mathbf{k}_3}\Bigg]\overline{u}_{N,\boldsymbol{\ell}}\Bigg)\notag\\
         & = 2\sum\limits_{|\boldsymbol{\ell}| \leq N}(1 + |\boldsymbol{\ell}|^2)^s\underbrace{Re\Big[i(\ell_{1}^3 + \ell_{1}\ell_{2}^2)|u_{N,\boldsymbol{\ell}}|^2}_{0}\Big] \notag\\
         &\qquad \qquad + 12Re\Bigg[i\sum\limits_{\boldsymbol{|\ell}| \leq N}(1 + |\boldsymbol{\ell}|^2)^s\Bigg(\sum\limits_{\substack{\boldsymbol{\ell} = \mathbf{k}_1 - \mathbf{k}_2 + \mathbf{k}_3\\\mathbf{k}_1,\mathbf{k}_2,\mathbf{k}_3 \leq N}}k_{3,1}u_{N,\mathbf{k}_1}\overline{u}_{N,\mathbf{k}_2}u_{N,\mathbf{k}_3}\overline{u}_{N,\boldsymbol{\ell}}\Bigg)\Bigg]\notag\\
         & = 12Re \left(i\sum\limits_{|\boldsymbol{\ell}|\leq N}(1 + |\boldsymbol{\ell}|^2)^s\sum\limits_{\substack{\boldsymbol{\ell} = \mathbf{k}_1 - \mathbf{k}_2 + \mathbf{k}_3\\\mathbf{k}_1,\mathbf{k}_2,\mathbf{k}_3 \leq N}}k_{3,1}u_{N,\mathbf{k}_1}\overline{u}_{N,\mathbf{k}_2}u_{N,\mathbf{k}_3}\overline{u}_{N,\boldsymbol{\ell}}\right).
    \end{align}

    Expressing \eqref{eq:truncated t derivative} in terms of an integral over $\T^2$, we obtain
    \begin{align}
        \frac{d}{dt}\|J^s\Pi_N u_N\|_{L^2}^2 = 12Re\int_{\T^2} J^s\Pi_N(|\Pi_Nu_N|^2(\partial_{x_1} \Pi_Nu_N))\overline{J^s \Pi_Nu_N}dz\nonumber.
    \end{align}
Then, adding and subtracting $12Re\left(|\Pi_Nu_N|^2\partial_{x_1}J^s\Pi_Nu_N\right)\overline{J^s\Pi_Nu_N}$, we have 
    \begin{align}\label{eq:prep energy bound}
    \frac{d}{dt}\|J^s\Pi_N u_N\|_{L^2}^2 &= 
        12Re\int_{\T^2} \left[J^s\left(|\Pi_Nu_N|^2\partial_{x_1}\Pi_Nu_N\right) - |\Pi_Nu_N|^2\partial_{x_1} J^s\Pi_Nu_N\right]\overline{J^s \Pi_Nu_N}dz \notag\\
        &\qquad + 12Re\int_{\T^2}|\Pi_Nu_N|^2\partial_{x_1} J^s \Pi_Nu_N \overline{J^s \Pi_N u_N}dz\nonumber\\
        & = 12Re\int_{\T^2} \left[J^s(|\Pi_Nu_N|^2\partial_{x_1} \Pi_Nu_N) - |\Pi_Nu_N|^2\partial_{x_1} J^s \Pi_Nu_N\right]\overline{J^s \Pi_Nu_N}dz \notag\\
        &\qquad + 6Re\int_{\T^2} |\Pi_Nu_N|^2\partial_{x_1}|J^s \Pi_Nu_N|^2dz\notag\\
        & = 12Re\int_{\T^2} \left[J^s(|\Pi_Nu_N|^2\partial_{x_1}\Pi_N u_N) - |\Pi_Nu_N|^2\partial_{x_1} J^s \Pi_N u_N\right]\overline{J^s \Pi_Nu_N}dz \notag\\
        &\qquad -6\int_{\T^2} \partial_{x_1}|\Pi_Nu_N|^2|J^s \Pi_Nu_N|^2dz.
    \end{align}
    where we integrated by parts in the second term to arrive at \eqref{eq:prep energy bound}.

Then, applying H\"older in each term, we have 
\begin{align*}
    \frac{d}{dt}\|J^s \Pi_N u_N\|_{L^2}^2 &\leq C\Bigg(\left\|J^s\left(|\Pi_Nu_N|^2\partial_{x_1}\Pi_Nu_N\right) - |\Pi_Nu_N|^2\partial_{x_1} J^s\Pi_Nu_N\right\|_{L^2(\T^2)}\left\|J^s\Pi_Nu_N\right\|_{L^2(\T^2)} \\
    &\qquad + \left\|\partial_{x_1} |\Pi_Nu_N|^2\right\|_{L^\infty(\T^2)}\|J^s \Pi_Nu_N\|_{L^2(\T^2)}^2\Bigg)
\end{align*}
Then, applying the Kato-Ponce commutator estimate in Lemma \ref{lem:KP} given by
\begin{align}\|J^s(fg) - fJ^sg\|_{L^2(\T^2)} \leq c\Big\{\|J^sf\|_{L^2(\T^2)}\|g\|_{L^\infty(\T^2)} + \big(\|f\|_{L^\infty(\T^2)} + \|\nabla f\|_{L^\infty(\T^2)}\big)\|J^{s-1}g\|_{L^2(\T^2)}\Big\},\end{align}
with $f = |\Pi_Nu_N|^2$ and $g = \partial_{x_1} \Pi_N u_N$, we obtain 
\begin{align*}
    \frac{d}{dt}\|J^s \Pi_Nu_N\|_{L^2}^2 &\leq C\Bigg[\Big(\|J^s|\Pi_Nu_N|^2\|_{L^2(\T^2)}\left\| \partial_{x_1}\Pi_Nu_N\right\|_{L^\infty(\T^2)} \\
    &\qquad + \bigg[\||\Pi_Nu_N|^2\|_{L^\infty(\T^2)} + \|\nabla |\Pi_Nu_N|^2\|_{L^\infty(\T^2)}\bigg]\left\|J^{s-1}\partial_{x_1} \Pi_Nu_N\right\|_{L^2(\T^2)}\Big)\|J^s\Pi_Nu_N\|_{L^2(\T^2)} \\
    &\qquad + \left\|\partial_{x_1}|\Pi_Nu_N|^2\right\|_{L^\infty(\T^2)}\|J^s \Pi_Nu_N\|_{L^2(\T^2)}^2\Bigg].
\end{align*}
Applying Lemma \ref{lem:sq} and using that $\|\Pi_Nu_N\|_{L^\infty(\T^2)} \lesssim \|J^{1 + \epsilon} \Pi_Nu_N\|_{L^2(\T^2)}$  we have
\begin{align}
    \frac{d}{dt}\|J^s \Pi_Nu_N\|_{L^2}^2 &\leq   C\Bigg[\bigg(\|J^s\Pi_Nu_N\|_{L^2}^2\|\partial_{x_1} \Pi_Nu_N\|_{L^\infty}\notag \\
    &\qquad + \bigg[\|J^s\Pi_N u_N\|_{L^2}^2
    + \|J^s\Pi_N u_N\|_{L^2}\|\nabla\Pi_N u_N\|_{L^\infty}\bigg]\|J^s \Pi_Nu_N\|_{L^2}\bigg)\|J^s \Pi_Nu_N\|_{L^2}\nonumber\\
    &\qquad \qquad + \|J^s\Pi_N u_N\|_{L^2}\|\nabla \Pi_Nu_N\|_{L^\infty}\|J^s \Pi_Nu_N \|_{L^2}^2\Bigg]\nonumber\\
    & \leq C\left[\|J^s \Pi_Nu_N\|_{L^2}^3\|\nabla \Pi_Nu_N\|_{L^\infty} + \|J^s\Pi_N u_N\|_{L^2}^4\right]
\end{align}
which is \eqref{eq:lem3.3}.

\end{proof}

Now we state a Strichartz estimate from \cite{Linares2019} which will be used below.

\begin{lemma}\label{lem:strichartz}
    Let \[W(t)u_0 = \int_{\R}\sum\limits_{\mathbf{k} \in \Z^d} e^{i(\mathbf{k}\cdot \mathbf{x} + ( k_1^3 + k_1k_2^2)t)}\hat{u}_0(\mathbf{k})\] and 
    \begin{align*}
            \widehat{Q^0g(\mathbf{k})}: = \1_{[0, 1)}(|\mathbf{k}|)\hat{g}(\mathbf{k}), \quad &\widehat{Q^jg(\mathbf{k})}:=\1_{[2^{j-1},2^j)}(|\mathbf{k}|)\hat{g}(\mathbf{k})
    \end{align*}
    as before.
     Then, for any $u_0 \in L^2(\T^2)$ and time interval $I \subset \R$ with $|I| \sim 2^{-2j}$ \[\|W(t)Q^j u_0\|_{L^2_IL^\infty} \lesssim 2^{-\frac{j}{3}}\|Q^ju_0\|_{L^2}\]
\end{lemma}
Lemma \ref{lem:strichartz} follows from Lemma 2.1 in \cite{Linares2019} by taking $N = 2^j$.  

We now leverage the integration in time with the Strichartz estimate given in Lemma \ref{lem:strichartz} in order to obtain a bound on $\nabla u_N$ in $L^1_tL^\infty(\T^2)$. 
\begin{proposition} \label{lem:lower s}
Let $s > \frac{5}{3}$. Then for $0 < T \leq 1$,
\begin{align}
    \|\nabla \Pi_Nu_N\|_{L^1([0,T];L^\infty_x)}\leq CT^{1/2}\left[\sup_{[0,T]}\|J^s\Pi_Nu_N(t)\|_{L^2_x} + T\sup_{[0,T]}\|J^s\Pi_Nu_N(t)\|_{L^2_x}^3\right].
\end{align}
with $u_N = \Phi_Nu_0$ where $\Phi_N$ is the flow corresponding to \eqref{eq:fin dim}. 
\end{proposition}

\begin{proof} Let $s > \frac{5}{3}$.  
Since the goal is the estimate $\|\nabla \Pi_Nu_N\|_{L^1([0,T];L^\infty_x)}$ we introduce $v_N: = \nabla \Pi_N u_N$ and $f_N := \Pi_N(|\Pi_Nu_N|^2\partial_{x_1} \Pi_Nu_N)$. Then, thanks to \eqref{eq:fin dim}, we can write the equation for $v_N$ as follows: 
\begin{align}\label{eq:v equation}
    \partial_t v_N + \partial_{x_1}\Delta v_N = \nabla f_N.
\end{align}

In order to estimate $\|v_N\|_{L^1([0,T];L^\infty)}$, let us fix $j \geq 1$ and subdivide the interval $[0,T]$ into $2^{2j}$  subintervals $I_\ell$ with $1 \leq \ell \leq 2^{2j}$ such that $|I_\ell| = 2^{-2j}T$. Applying the triangle inequality, we sum over the time-localized contributions, that is, each subinterval $I_\ell$. Then, we apply H\"older to obtain
\begin{align}
    \|Q^j(v_N)\|_{L^1([0,T];L^\infty)} &\leq \sum\limits_{\ell=1}^{2^{2j}}\|Q^j(v_N)\|_{L^1(I_\ell;L^\infty)}\nonumber\\
    & \leq \sum\limits_{\ell = 1}^{2^{2j}}|I_\ell|^{1/2}\|Q^j(v_N)\|_{L^2(I_\ell,L^\infty)}\notag \\
    &= T^{1/2}2^{-
    j}\sum\limits_{\ell = 1}^{2^{2j}} \|Q^j(v_N)\|_{L^2(I_\ell;L^\infty)}\label{eq:L1 estimate 1}.
\end{align}
Thus, we have shifted our task to estimating $Q^j(v_N)$ in $L^2(I_\ell, L^\infty)$.
 
Let $I_\ell = [a_\ell, a_{\ell + 1}]$. By Duhamel's  formula corresponding to \eqref{eq:v equation}, for $t \in I_\ell$, 
\begin{align}v_N(t) = W(t - a_\ell)(v_N(a_\ell)) + \int_{a_\ell}^t W(t-\tau)\nabla f_N(\tau)d\tau\end{align} so that \begin{align}Q^j(v_N(t)) = W(t-a_\ell)(Q^j(v_N(a_\ell))) + \int_{a_\ell}^tW(t-\tau)\nabla Q^j(f_N(\tau))d\tau.\end{align}

Now we apply the Strichartz estimate from Lemma \ref{lem:strichartz} to obtain
\begin{align}
    \|Q^j(v_N(t))\|_{L^2_{I_\ell}L^\infty} &\leq \|W(t-a_\ell)(Q^j(v_N(a_\ell)))\|_{L^2_{I_\ell}L^\infty} + \left\lvert\left\lvert\int_{a_\ell}^t W(t-\tau)\nabla Q^j(f(\tau))d\tau\right\rvert\right\rvert_{L^2_{I_\ell}L^\infty}\nonumber\\
    &\leq C2^{-\frac{j}{3}} \|Q^jv_N(a_\ell)\|_{L^2} + C 2^{-\frac{j}{3}}2^{j}\int_{a_\ell}^{a_{\ell + 1}}\|Q^jf_N(\tau)\|_{L^2}d\tau.\label{eq:L2 estimate}
\end{align}
Combining \eqref{eq:L1 estimate 1} and \eqref{eq:L2 estimate}, 
\begin{align*}
    \|Q^j(v_N(t))\|_{L^1([0,T]; L^\infty)} &\leq T^{1/2}2^{-j}\sum\limits_{\ell=1}^{2^{2j}}\|Q^j(v_N)\|_{L^2(I_\ell; L^\infty)}\nonumber\\
    & \leq CT^{1/2}2^{-j}\Bigg[\left(\sum\limits_{\ell =1}^{2^{2j}} 2^{-\frac{j}{3}}\|Q^jv_N(a_\ell)\|_{L^2}\right) + \sum\limits_{\ell = 1}^{2^{2j}}2^{\frac{2}{3}j}\int_{a_\ell}^{a_{\ell + 1}}\|Q^jf_N(\tau)\|_{L^2}d\tau\Bigg]\\
    & \leq  CT^{1/2}2^{-j}\Big[2^{-\frac{j}{3}}2^{2j}\sup_{t \in [0,T]}\|Q^jv_N(t)\|_{L^2} + 2^{\frac{2}{3}j}\int_0^T \|Q^jf_N(\tau)\|_{L^2}\,d\tau\Big]\\
    & = CT^{1/2}\Big[2^{\frac{2}{3}j}\sup\limits_{t \in [0,T]}\|Q^j v_N(t)\|_{L^2} + 2^{-\frac{j}{3}}\int_0^T \|Q^j f_N(\tau)\|_{L^2}\,d\tau\Big]
\end{align*}

Since $j \geq 1$, we have 
\begin{align}\label{eq:L1 estimate 2}
     \|Q^j(v_N(t))\|_{L^1([0,T]; L^\infty)} & \leq CT^{1/2}2^{\frac{2}{3}j}\Bigg(\sup\limits_{t \in [0,T]}\|Q^jv_N(t)\|_{L^2} + \int_0^T \|Q^j f_N(\tau)\|_{L^2}\,d\tau \Bigg)
\end{align}
Let $s_0 = \frac{2}{3} + 2\epsilon$. 
It follows from \eqref{eq:L1 estimate 2} that
\begin{align*}
    \|Q^j(v_N)\|_{L^1[0,T],L^\infty)} \leq CT^{1/2}2^{-2\epsilon j}\sup\limits_{t \in[0,1]}\|J^{s_0}(v_N(t))\|_{L^2} + CT^{1/2}2^{-2\epsilon j}\int_0^T\|J^{s_0}(f_N)(\tau)\|_{L^2}d\tau.
\end{align*}
Then, summing in $j$, we have
\begin{align}\label{eq:full L1 estimate}
    \|v_N\|_{L^1([0,T];L^\infty)} \leq CT^{1/2}\sup\limits_{t \in [0,T]}\|J^{s_0}(v_N(t))\|_{L^2} + CT^{1/2}\int_0^T \|J^{s_0}(f_N)(\tau)\|_{L^2}d\tau,
\end{align}
which, since $v_N = \nabla \Pi_Nu_N$, implies  
\begin{align}\label{eq:estimate 1}
    \left\|\nabla \Pi_Nu_N\right\|_{L^1([0,T];L^\infty)} &\leq CT^{1/2}\sup\limits_{t \in [0,T]}\|J^{s_0 + 1}\Pi_Nu_N(t)\|_{L^2}\notag\\
    &\qquad + CT^{1/2}\int_0^T \left\|J^{s_0}\left(\Pi_N|\Pi_Nu_N|^2\partial_{x_1}\Pi_Nu_N\right)(\tau)\right\|_{L^2}d\tau.
\end{align}

Now we apply the Kato-Ponce estimate given in Lemma \ref{lem:KP} given by 
 \begin{align}
     \|J^\sigma(fg) - fJ^\sigma g \|_{L^2} \leq C\|J^\sigma f\|_{L^2}\|g\|_{L^\infty}
 \end{align}
 with $f = \partial_{x_1} \Pi_Nu_N$ and $g = |\Pi_Nu_N|^2$ to obtain
 \begin{align}\label{eq:nonlinearity estimate}
\left\|J^{s_0}\left(|\Pi_Nu_N|^2\partial_{x_1}\Pi_Nu_N\right)(s)\right\|_{L^2} &\leq \left\|\left(\partial_{x_1}\Pi_Nu_N\right)J^{s_0}(|\Pi_Nu_N|^2)\right\|_{L^2} \notag\\
&\qquad \qquad + \left\|\partial_{x_1} J^{s_0}\Pi_Nu_N\right\|_{L^2}\|\Pi_Nu_N\|_{L^\infty}^2.
\end{align}

Estimating the first term on the right side of \eqref{eq:nonlinearity estimate} using Sobolev embedding and Lemma \ref{lem:sq} as well as the fact that $s > 1$,
\begin{align}\label{eq:nonlinear estimate 2}
    \|\left(\partial_{x_1}\Pi_Nu_N\right) J^{s_0}(|\Pi_Nu_N|^2)\|_{L^2} &\leq \|\nabla \Pi_Nu_N\|_{L^2}\|J^{s_0}(|\Pi_Nu_N|^2)\|_{L^\infty}\notag\\
    & \leq\|J^{s}\Pi_N u_N\|_{L^2}\|J^{s_0 + 1 + \epsilon}(|\Pi_Nu_N|^2)\|_{L^2}\notag\\
    & \leq \|J^{s} \Pi_Nu_N\|_{L^2}^3,
\end{align}
which holds for $s \geq s_0 + 1 + \epsilon = \frac{5}{3} + 3\epsilon$, using Lemma \ref{lem:sq}.

Returning to \eqref{eq:nonlinearity estimate}, thanks to  \eqref{eq:nonlinear estimate 2} and recalling that $\|\Pi_Nu_N\|_{L^\infty} \leq \|J^s \Pi_Nu_N\|_{L^2}$,
\begin{align}\label{eq:nonlinear estimate 3}
    \left\|J^{s_0}(|\Pi_Nu_N|^2\partial_{x_1}\Pi_Nu_N)\right\|_{L^2} \lesssim \|J^s \Pi_Nu_N \|_{L^2}^3.
\end{align}
Therefore, combining \eqref{eq:estimate 1} and \eqref{eq:nonlinear estimate 3}, we obtain the desired estimate
\begin{align}
    \|\nabla \Pi_Nu_N\|_{L^1([0,T];L^\infty)} \leq C T^{1/2} \left[\sup\limits_{t \in[0,T]}\|J^s\Pi_Nu_N(t)\|_{L^2} + T\sup\limits_{t \in [0,T]}\|J^s\Pi_Nu_N\|_{L^2}^3\right],
\end{align}
for $s > \frac{5}{3}.$
\end{proof}

\begin{remark}
    We note that in the proof of Lemma \ref{lem:lower s} we did not make use of the truncation of the finite dimensional problem \eqref{eq:fin dim}. Thus, we can use an identical argument for $u$ that solves the mZK equation \eqref{eq:mZK}, to show that 
    \begin{align}\label{eq:L1Linf for mZK}
        \|\nabla u\|_{L^1([0,T];L^\infty_x)} \leq CT^{1/2}\Big[\sup\limits_{[0,T]}\|J^su(t)\|_{L^2_x} + T\sup\limits_{[0,T]}\|J^su(t)\|_{L^2_x}^3\Big].
    \end{align}
\end{remark}

\begin{lemma} \label{lem:local time}
   Let $s > \frac{5}{3}$. Then there exists a $T > 0$ depending on $s$ and $\|J^su_0\|_{L^2}$ such that 
    \begin{align}\label{eq:N apriori bound}
            \sup_{[0,T]}\|J^s\Pi_Nu_N(t)\|_{L^2} \leq 2\|J^s \Pi_Nu_0\|_{L^2}.
    \end{align}
    where $u_N(t) = \Phi_N(t)u_0$ and $\Phi_N$ is the flow corresponding to \eqref{eq:fin dim}.
\end{lemma}
\begin{proof}
First, we note that for $\|J^s\Pi_N u_0\|_{L^2} = 0$, by the uniqueness of the finite dimensional problem we have $\|J^s\Pi_N u_N(t)\|_{L^2} = 0$ and the inequality \eqref{eq:N apriori bound} is trivial. Now we assume $\|J^s\Pi_N u_0\|_{L^2} > 0$.

 Define $Y(t) = \sup\limits_{\tau \in [0,t]}\|J^s\Pi_Nu_N(\tau)\|_{L^2}^2$. Fix $T > 0$. We consider the two cases: 

    \textbf{Case 1:} ($\bm{Y(T) < 2\|J^s \Pi_Nu_0\|_{L^2}^2}$) The bound \eqref{eq:N apriori bound} is obtained.

    \textbf{Case 2:} ($\bm{Y(T) \geq 2\|J^s \Pi_Nu_0\|_{L^2}^2}$) Let $0 < \widetilde{T} \leq T$ be the first time that \[\|J^s\Pi_Nu_N(\widetilde{T})\|_{L^2}^2 = 2\|J^s \Pi_Nu_0\|_{L^2}^2.\] We will derive a contradiction for well-chosen $T$ using Lemma \ref{lem:energy bound} and Proposition \ref{lem:lower s}.

    Integrating \eqref{eq:lem3.3} on the interval $[0,\widetilde{T}]$, we have 
    \begin{align}
        \|J^s\Pi_Nu_N(\widetilde{T})\|_{L^2}^2 - \|J^s\Pi_Nu_0\|_{L^2}^2 &\leq C\int_0^{\widetilde{T}}\Big(\|J^s\Pi_Nu_N\|^3_{L^2}\|\|\nabla \Pi_Nu_N\|_{L^\infty} + \|J^s \Pi_Nu_N\|_{L^2}^4\Big)dt\notag\\
        & \leq C[Y(\widetilde{T})]^{3/2}\|\nabla \Pi_Nu_N\|_{L^1([0,\widetilde{T}];L_x^\infty)} + \widetilde{T}[Y(\widetilde{T})]^2.
    \end{align}
    Then, thanks for Proposition \ref{lem:lower s}, 
    \begin{align}
    \|J^s\Pi_Nu_N(\Tilde{T})\|_{L^2}^2 - \|J^s\Pi_Nu_0\|_{L^2}^2
        & \leq C\widetilde{T}^{1/2}[Y(\widetilde{T})]^{3/2}\Big[[Y(\widetilde{T})]^{1/2} + \widetilde{T}[Y(\widetilde{T})]^{3/2}\Big] + \widetilde{T}[Y(\widetilde{T})]^2\notag\\
        & \leq C\widetilde{T}^{1/2}[Y(\widetilde{T})]^{2} + C\widetilde{T}^{3/2}[Y(\widetilde{T})]^{3} + C\widetilde{T}[Y(\widetilde{T})]^2
    \end{align}
    By assumption, $Y(\widetilde{T}) = 2\|J^s\Pi_Nu_0\|_{L^2}^2$ and $0 \leq \widetilde{T} \leq T$, so 
    
    \begin{align}
        \|J^s\Pi_Nu_0\|_{L^2}^2 \leq 4CT^{1/2}\|J^s \Pi_Nu_0\|_{L^2}^4 + 8CT^{3/2}\|J^s\Pi_N u_0\|_{L^2}^6 + 4CT\|J^s\Pi_Nu_0\|_{L^2}^4.
    \end{align}

    Dividing by $\|J^s \Pi_Nu_0\|_{L^2}^2$ (recall $\|J^s \Pi_N u_0\|_{L^2} > 0)$, we have 
    \begin{align}\label{eq:get contradiction}
        1 &\leq 4CT^{1/2}\|J^s \Pi_Nu_0\|_{L^2}^2 + 8CT^{3/2}\|J^s \Pi_Nu_0\|_{L^2}^4 + 4CT\|J^s \Pi_Nu_0\|_{L^2}^2\notag\\
        & \leq 4CT^{1/2}\|J^s u_0\|_{L^2}^2 + 8CT^{3/2}\|J^s u_0\|_{L^2}^4 + 4CT\|J^s u_0\|_{L^2}^2.
    \end{align}
where we have used that $\|J^s \Pi_N u_0\|_{L^2} \leq \|J^s u_0\|_{L^2}$ to obtain \eqref{eq:get contradiction}. 

    We choose $T$ small enough in terms of $\|J^s u_0\|_{L^2}$ such that the right hand side of \eqref{eq:get contradiction} is smaller than $1/2$ and we have obtained a contradiction. 
\end{proof}

We obtain the following as an immediate corollary of the previous lemma.
\begin{corollary} \label{cor:local time}
   Let $s > \frac{5}{3}$. Then there exists a $T > 0$ depending on $s$ and $\|J^su_0\|_{L^2}$ such that 
    \begin{align}\label{eq:apriori bound}
            \sup_{[0,T]}\|J^s\Pi_Nu_N(t)\|_{L^2} < 2\|J^s u_0\|_{L^2}.
    \end{align}
    where $u_N(t) = \Phi_N(t)u_0$ and $\Phi_N$ is the flow corresponding to \eqref{eq:fin dim}.
\end{corollary}

\subsection{Towards the proof of Theorem \ref{thm:lwp} and the proof itself}

We wish to pass to the limit of the finite dimensional problem. To do this we will need the following lemma. We are motivated by the work \cite{Kenig2025}.

\begin{lemma}\label{lem: towards existence}
Let $\{u_{N,0}\}$ be a sequence in $H^s$, with $s > \frac{5}{3}$, such that 
\begin{align}
    u_{N,0} \overset{N\to \infty}{\longrightarrow} u_0\text{ in $L^2$ and }\sup_{N}\|u_{N,0}\|_{H^s} = S < \infty.
\end{align}
Then the sequence $\{\Phi_N(t)u_{N,0}\} \subset C([-T,T],H^s)$ of solutions to \eqref{eq:fin dim} in the sense of Definition \ref{def:strong solution} with initial conditions $\{u_{N,0}\}$ is a Cauchy sequence in $C([-T,T], L^2)$. 
\end{lemma}
\begin{proof}
    From the definition of the flow associated to \eqref{eq:fin dim} we have the following splitting: 
    \begin{align*}
        \Phi_Nu_{N,0} = \Pi_N\Phi_Nu_{N,0} + \Pi_{> N}W(t)u_{N,0},
    \end{align*}
    where $W(t)$ is the propagator associated with the linear equation given in \eqref{eq:W}. 
    
    Assume $N \leq M$. Then we have
    \begin{align}
        \|\Pi_{> M}W(t)u_{M,0} - \Pi_{> N}W(t)u_{N,0}\|_{L^2} &\leq \|\Pi_{>M}W(t)u_{M,0} - \Pi_{>M}W(t)u_{N,0}\|_{L^2} \notag\\
        &\qquad + \|\Pi_{>M}W(t)u_{N,0} - \Pi_{>N}W(t)u_{N,0}\|_{L^2}\notag\\
        & \leq \|u_{M,0} - u_{N,0}\|_{L^2} + 2\|\Pi_{> N}u_{N,0}\|_{L^2}\notag\\
        & \leq \|u_{M,0} - u_{N,0}\|_{L^2} + 2N^{-s}\|\Pi_{> N}u_{N,0}\|_{H^s}.
    \end{align}
    Since $u_{N,0} \to u_0$ in $L^2$ and $\|u_{N,0}\|_{H^s} < S$ is uniformly bounded in $N$, we have that $\{\Pi_{>N}W(t)u_{N,0}\}$ is a Cauchy sequence in $C([-T,T], L^2)$. Thus, it is sufficient to show that \begin{align}w_N(t) := \Pi_N\Phi_Nu_{N,0}\end{align} is a Cauchy sequence in $C([-T,T], L^2)$. First, note that $w_N(t)$ solves \eqref{eq:fin dim} with initial data $\Pi_Nu_{N,0}$. So, $w_N - w_M$ solves 
    \begin{align}
        (\partial_t + \partial_{x_1}\Delta)(w_N - w_M) = \Pi_N(|w_N|^2\partial_{x_1}w_N) - \Pi_M(|w_M|^2\partial_{x_1}w_M).
    \end{align}

    Using this equation 
    we can express $\frac{d}{dt}\|w_N - w_M\|_{L^2}^2$ as follows
    \begin{align}\label{eq:L2 diff}
        \frac{d}{dt}\|w_N - w_M\|_{L^2}^2 &= 12Re \int \Big[\Pi_N(|w_N|^2\partial_{x_1}w_N) - \Pi_M(|w_M|^2\partial_{x_1}w_M)\Big]\overline{(w_N - w_M)}
    \end{align}

By first applying H\"older's inequalty and then the Sobolev embedding $H^s \subset L^\infty$ and $N \leq M$, we have 
\begin{align}\label{eq:cauchy time estimate}
    \frac{d}{dt}\|w_N - w_M\|_{L^2}^2 &\leq C\big\|\Pi_N(|w_N|^2\partial_{x_1}w_N) - \Pi_M(|w_M|^2\partial_{x_1}w_M)\|_{L^2}\|w_N - w_N\|_{L^2}\notag\\
    & \leq C\Big(\||w_N|^2\partial_{x_1}w_N\|_{L^2} + \||w_M|^2\partial_{x_1}w_M\|_{L^2}\Big)\|w_N - w_M\|_{L^2}\notag\\
    & \leq C\Big(\|w_N\|_{L^\infty}^2\|\partial_{x_1}w_N\|_{L^2} + \|w_M\|_{L^\infty}^2\|\partial_{x_1}w_M\|_{L^2}\Big)N^{-s}\|w_N - w_M\|_{L^2}\notag\\
    & \leq C\Big(\|w_N\|_{H^s}^3 + \|w_M\|_{H^s}^3\Big)N^{-s}(\|w_N\|_{H^s} + \|w_M\|_{H^s})
\end{align}

Let $\tau > 0$. Returning to \eqref{eq:cauchy time estimate} and integrating  $\frac{d}{dt}\|w_N(t) - w_M(t)\|_{L^2}^2$ on the interval $[0,\tau]$ we have 
\begin{align}\label{eq:towards cauchy}
    &\sup\limits_{t \in [0,\tau]}\|w_N(t) - w_M(t)\|_{L^2}^2 \notag\\
     &\leq \|w_N(0) - w_M(0)\|_{L^2}^2 + C\tau\sup_{t \in [0,\tau]}(\|w_N(t)\|_{H^s}^3 + \|w_M\|_{H^s}^3)(\|w_N(t)\|_{H^s} + \|w_M(t)\|_{H^s})N^{-s}\nonumber\\
\end{align}

Since $w_N$ solves \eqref{eq:fin dim}, Lemma \ref{lem:local time} implies \[\|w_N(t)\|_{H^s}, \|w_N(t)\|_{H^s} < 2S \quad \forall \, t \in [0,\tau],\] which together with Lemma \ref{lem:lower s} and \eqref{eq:towards cauchy} gives
\begin{align}
    \sup\limits_{t \in [0,\tau]}\|w_N(t) - w_M(t)\|_{L^2}^2 & \leq \|w_N(0) - w_M(0)\|_{L^2}^2 +  C\tau S^4N^{-s}\nonumber 
\end{align}
We choose $\tau$ small enough, that is $\tau \leq \min\left(\left(\frac{1}{4CS^2}\right)^2, \left(\frac{1}{4CS^4}\right)^{2/3}\right)$, so that the last term on the right hand side can be absorbed on the left hand side. Then we have
\begin{align}\label{eq:cauchy}
    \sup\limits_{t \in [0,\tau]}\|w_N(t) - w_M(t)\|_{L^2}^2 & \leq C\|w_N(0)- w_M(0)\|_{L^2}^2 + C\tau S^4N^{-s}.
\end{align}
Since $u_{N,0} \to u_0$ in $L^2$, we get that $w_N(0) = \Pi_Nu_{N,0}$ is a Cauchy sequence in $L^2$. Therefore, from \eqref{eq:cauchy} we obtain that $w_N$ is a Cauchy sequence in $C([0,\tau],L^2)$. In particular $w_N(\tau)$ converges in $L^2$. Thus, we can iterate this argument above on the interval $[\tau,2\tau]$ to get that $w_N$ is a Cauchy sequence on $C([\tau,2\tau],L^2)$. Repeating this argument $\left[\frac{T}{\tau}\right] + 1$ times, we obtain that $w_N$ is a Cauchy sequence in $C([0,T],L^2)$. By a similar argument, we get that $w_N$ is a Cauchy sequence in $C([-T,0],L^2)$ and we obtain the result.

For the proof of continuity in time see \cite{Cazenave2003}.
\end{proof}

With this lemma in hand, we can prove the local existence of a strong solution!

\begin{proposition}[Existence]\label{prop:existence}
    Let $s > \frac{5}{3}$. Then for every $u_0 \in H^s$ there exists a strong solution $u \in C([-T,T], H^s)$ to \eqref{eq:mZK} where $T$ is given in Lemma \ref{lem:local time}. 
\end{proposition}

\begin{proof}
Let $T$ be given by Lemma \ref{lem:local time}. 
    Consider the sequence ${u_N}$ of solutions to \eqref{eq:fin dim} with initial conditions $u_0$ independent of $N$. By Corollary \ref{cor:local time} we know that
    \begin{align}
        \sup_{t \in [0,T]}\|\Pi_Nu_N(t)\|_{H^s} \leq 2\|u_0\|_{H^{s}}. 
    \end{align}

    Moreover, since  $\frac{d}{dt}\|J^s\Pi_{>N}u_N\|_{L^2} = 0$, we have that
    \begin{align*}
        \|\Pi_{>N} u_N\|_{H^s} = \|\Pi_{>N}u_0\|_{H^s} \leq \|u_0\|_{H^s}. 
    \end{align*} 

    Thus,
    \begin{align}\label{eq:unif bound}
        \sup_{t \in [0,T]}\|u_N(t)\|_{H^s} \leq 3\|u_0\|_{H^s}.
    \end{align}
    Then, by Lemma \ref{lem: towards existence}, we have that there exists a $u \in C([-T,T], L^2)$ such that $u_N \to u$ in $C([-T,T],L^2)$. Combining the bound \eqref{eq:unif bound} with this convergence result using an interpolation argument, we actually have that 
    \begin{align}\label{conv in s'}
        u_N \to u \text{ in } C([-T,T], H^{s'})~\forall~0 \leq s' < s. 
    \end{align}
    Pick $s'$ such that $\frac{5}{3} < s' < s$. The integral equation of the truncated problem \eqref{eq:fin dim} is written as  
    \begin{align}\label{eq:fin int eq}
        u_N(t) = W(t)u_0 + \int_0^t W(t - \tau)\Pi_N(|\Pi_N u_N|^2\partial_{x_1}\Pi_N u_N(\tau))\,d\tau.
    \end{align}
    We claim that  
    \begin{align}\label{eq:nonlin convergence}
        \Pi_N(|\Pi_N u_N|^2\partial_{x_1}\Pi_N u_N) \to |u|^2\partial_{x_1}u \text{ in }C([-T,T], L^2).
    \end{align}
    In order to prove \eqref{eq:nonlin convergence}, we proceed as follows. 
    \begin{align*}
        \|\Pi_N(|\Pi_N u_N|^2\partial_{x_1}\Pi_N u_N) - |u|^2\partial_{x_1}u\|_{L^2} \leq \big\|\Pi_N(|\Pi_N u_N|^2\partial_{x_1}\Pi_N u_N - |u|^2\partial_{x_1}u)\big\|_{L^2} + \|\Pi_{>N}|u|^2\partial_{x_1}u\|_{L^2}
    \end{align*}
    Then, since 
    \begin{align*}
        \||u|^2\partial_{x_1}u\|_{L^2} \leq \|u\|_{L^\infty}^2\|\partial_{x_1}u\|_{L^2} \leq \|u\|_{H^{s'}}^3
    \end{align*}
     and $u \in H^{s'}$, we have that $|u|^2\partial_{x_1}u \in L^2$. Thus, the second term goes to zero as $N\to \infty$. 

     We add and subtract $\Pi_N(|u|^2\partial_{x_1}\Pi_N u_N)$ in the first term to obtain
     \begin{align*}
         &\big\|\Pi_N(|\Pi_N u_N|^2\partial_{x_1}\Pi_N u_N - |u|^2\partial_{x_1}u)\big\|_{L^2} \\
         &\qquad \leq \|(|\Pi_N u_N|^2 - |u|^2)\partial_{x_1}\Pi_N u_N\|_{L^2} + \||u|^2\partial_{x_1}(\Pi_N u_N - u)\|_{L^2}\\
         & \qquad \leq \|(\Pi_Nu_N - u)\overline{\Pi_Nu_N} + u(\overline{\Pi_Nu_N - u})\|_{L^\infty}\|\partial_{x_1}\Pi_N u_N\|_{L^2} + \|u\|_{L^\infty}^2\|\partial_{x_1}(\Pi_Nu_N - u)\|_{L^2}\\
          &\qquad \leq \|\Pi_Nu_N - u\|_{H^{s'}}(\|\Pi_Nu_N\|_{H^{s'}} + \|u\|_{H^{s'}})\|\Pi_N u_N\|_{H^{s'}} + \|u\|_{H^{s'}}^2\|\Pi_Nu_N - u\|_{H^{s'}}.
     \end{align*}
     Thus, since $u_N \to u$ in $H^{s'}$, the first terms also goes to zero as $N\to\infty$.

    Passing to the limit in \eqref{eq:fin int eq} in $L^2$ for each $t \in [0,T]$, we arrive at the integral equation for \eqref{eq:mZK}, so $u$ is a strong solution of \eqref{eq:mZK} in $C([-T,T],H^{s'})$. At this point, we can use the Bona-Smith argument \cite{BonaSmith1975} to conclude that $u \in C([-T,T],H^s)$. We will demonstrate this process below by proving 3 claims.

    \begin{enumerate}
        \item \textbf{Claim 1}: If $t_n \overset{n\to\infty}{\to} \overline{t}$, then $u(t_n) \overset{n\to \infty}{\rightharpoonup}u(\overline{t})$ in $H^s$, i.e. $u \in C_{w}([-T,T],H^s)$

        To prove this, we first note that since $u_N \to u$ in $C([-T,T],H^{s'})$, we have that $u \in C([-T,T],H^{s'})$. Thus $\|u(t_n) - u(\overline{t})\|_{H^{s'}} \to 0$ for $\frac{5}{3} < s' < s$ and by \eqref{eq:unif bound} we know that $\sup_n\|u(t_n)\|_{H^{s}} < \infty$. So, we can extract a weakly convergence subsequence $\{t_{n_k}\}_{k}$ such that $u(t_{n_k}) \rightharpoonup \overline{u}$ in $H^s$. Since  $u(t_n) \rightharpoonup u(\overline{t})$ in $H^{s'}$, by uniqueness of the limit in the sense of distributions, we have that $\overline{u} = u(\overline{t})$ and every convergent subsequence converges to $u(\overline{t})$. The proof of Claim 1 follows. 
        \item \textbf{Claim 2}: $u(t)$ is continuous at $t = 0$ in the strong $H^s$ topology. 

        To prove this, we let $t_n \to 0$ and show $\|u(t_n)\|_{H^s} \to \|u(0)\|_{H^s}$. 

        By Claim 1 and lower semicontinuity of the $H^s$ norm it holds that
        \begin{align}\label{eq:claim2upperbound}
            \|u(0)\|_{H^s} \leq \lim\inf\limits_{n\to \infty}\|u(t_n)\|_{H^s}. 
        \end{align}
        Using the energy estimate given in Lemma \ref{lem:energy bound}, 
        \begin{align*}
            \left|\frac{d}{dt}\|u_N(t)\|_{H^s}^2\right| \leq C\|\Pi_Nu_N(t)\|_{H^s}^3\|\Pi_N u_N(t)\|_{W^{1,\infty}} + \|\Pi_Nu_N(t)\|_{H^s}^4
        \end{align*}
        which, after integrating in time, yields 
        \begin{align*}
            \left|\|u_N(t_n)\|_{H^s}^2 - \|u_N(0)\|_{H^s}^2\right| &\leq C\|\Pi_Nu_N\|_{L^\infty([-T,T],H^s)}^3\int_0^{t_n}\|\Pi_N u_N(t)\|_{W^{1,\infty}}dt\\
            & \qquad + Ct_n\|\Pi_Nu_N(t)\|_{L^\infty([-T,T],H^s)}^4
        \end{align*}

        Then, applying Lemma \ref{lem:lower s}, we have 
        \begin{align}
            &\Big|\|u_N(t_n)\|_{H^s}^2 - \|u_N(0)\|_{H^s}^2\Big| \notag\\
            &\qquad \leq  C\|\Pi_Nu_N\|_{L^\infty([-T,T],H^s)}^3\left(t_n^{1/2}\|\Pi_N u_N\|_{L^\infty([-T,T],H^s)} + t_n^{3/2}\|\Pi_N u_N\|^3_{L^\infty([-T,T],H^s)}\right)\notag\\
            & \qquad \qquad  + Ct_n\|\Pi_Nu_N(t)\|_{L^\infty([-T,T],H^s)}^4
        \end{align}
        By \eqref{eq:unif bound}, this implies 
        \begin{align}\label{eq:bound on sequence}
        \|u_N\|_{H^s}^2 \leq \|u_N(0)\|_{H^s}^2 + C(|t_n|^{1/2}+ |t_n|^{3/2}) + C|t_n|
        \end{align}
        where $C = C(\|u_0\|_{H^s})$ depends on the initial data.
        
        Since $\|u_N(t)\|_{H^s}$ is uniformly bounded, for any subsequence $\{u_{N_k}\}_{k}$, we can extract a weakly convergent subsequence $u_{N_{k,j}}(t) \rightharpoonup v(t)$. We can consider this subsequence in $H^{s'}$ as well. Since $u_N(t) \to u(t)$ strongly in $H^{s'}$, then $u_{N_{k,j}}(t) \rightharpoonup u(t)$ in $H^{s'}$. That is, $u(t) = v(t)$. Thus, we can upgrade this convergence to the whole sequence, so $u_N(t) \rightharpoonup u(t)$ in $H^s$. 

        Then, thanks to \eqref{eq:bound on sequence}, we have 
        \begin{align}
            \|u(t_n)\|_{H^s}^2 &\leq \lim\inf\limits_{N \to \infty}\|u_N(t_n)\|_{H^s}^2\nonumber\\
            & \leq \lim\inf\limits_{N \to \infty}\Big(\|u(0)\|_{H^s}^2 + C(|t_n|^{1/2} + |t_n|^{3/2}) + C|t_n|\Big)\nonumber\\
            & = \|u(0)\|_{H^s}^2 + C(|t_n|^{1/2} + |t_n|^{3/2}) + C|t_n|.
        \end{align}
        Thus, since $t_n \to 0$,
        \begin{align}\label{eq:claim2lowerbound}
            \lim\sup_{n\to \infty}\|u(t_n)\|_{H^s}^2 \leq \|u(0)\|_{H^s}^2.
        \end{align}
        Combining \eqref{eq:claim2upperbound} with \eqref{eq:claim2lowerbound},
         Claim 2 is proved.
        \item \textbf{Claim 3}: $u(t)$ is continuous in the $H^s$ topology for any $t \in [-T,T]$. 

        To prove this, replace the initial data of the truncated initial value problem with $u_N(\overline{t}) = u(\overline{t})$ and repeat the argument for Claim 2. 
    \end{enumerate}
\end{proof}
\begin{proposition}[Uniqueness]
Let $s > \frac{5}{3}$ and $T > 0$. Suppose $u,v \in C([0,T],L^2)\cap L^\infty([0,T],H^s)$ are two solutions of \eqref{eq:mZK} with $u(0) = v(0)$. Then $u(t)= v(t)~\forall~ t\in [-T,T]$.    
\end{proposition}

\begin{proof}

 Let $u$ be a strong solution of the initial value problem $\eqref{eq:mZK}$. Let $\theta$ be a smooth and compactly supported function on $\R^2$, and we define on $\T^2$ for $\ve \in (0,1)$ the function $\theta_{\ve}$ such that $\hat{\theta}_{\ve}(\mathbf{n}) = \hat{\theta}(\ve \mathbf{n})$ and $\hat{\theta}(0) = 1$. We abuse notation and call $\theta_\ve$ the associated multiplier operator. Then if we define $u_\ve(t): = \theta_\ve u(t)$ we write 
    \begin{align}
        u_\ve(t) = W(t)\theta_\ve u_0 + \int_0^tW(t - t')\theta_\ve(|u|^2\partial_{x_1}u)(\tau)d\tau
    \end{align}
    in the pointwise sense. So we have 
    \begin{align}
        \begin{cases}
            \partial_t u_\ve + \partial_{x_1}\Delta u_\ve = \theta_\ve(|u|^2\partial_{x_1}u), & (t,x) \in \R\times \T^2, \\
            u_\ve(0) = \theta_\ve u_0.
        \end{cases}
    \end{align}
    Repeating the proof of Lemma \ref{lem:lower s}, we have 
    \begin{align}
        \sup_{\ve}\|\partial_{x_1} u_\ve\|_{L^1([-T,T],L^\infty} = C < \infty,
    \end{align}
    with $C$ is independent of $\ve$ and $T$ is small enough. We also note that $u_\ve$ converges to $u$ in $C([-T,T], H^s)$ as $\ve \to 0$. We claim 
    \begin{align}\label{eq:W1inf bound}
        \|\partial_{x_1}u\|_{L^1([-T,T], L^\infty)} \leq C
    \end{align} also. To prove this, pick $p \geq 1$ and $g$ such that $\|g\|_{L^\infty([-T,T], L^{q}_x)} \leq 1$ where $q$ is the H\"older conjugate of $p$. Then
    \begin{align*}
        \left|\int_{-T}^T\int_{\T^2} \partial_{x_1}u\overline{g}dxdt \right| &\leq \int_{-T}^T\int_{\T^2} |\partial_{x_1}u||g|dxdt\\
        & \leq \lim\inf\limits_{\ve \to 0}\int_{-T}^T\int_{\T^2} |\partial_{x_1}u_\ve||g|dxdt\\
        & \leq \int_{-T}^T\|\partial_{x_1}u_\ve\|_{L^p(\T^2)}\|g\|_{L^q(\T^2)}dt\\
        & \leq C\int_{-T}^T\|\partial_{x_1}u_\ve\|_{L^\infty}\\
        & \leq C. 
    \end{align*}
    As a consequence
    \begin{align}
        \int_{-T}^T\left(\int_{\T^2}|\partial_{x_1}u|^pdx\right)^{1/p}dt \leq C. 
    \end{align}
    Recall that for any function $f$ we have 
    \begin{align}
        \|f\|_{L^\infty(\T^2)} = \lim\limits_{p \to \infty}\left(\frac{1}{(2\pi)^2}\int_{\T^2} |f|^pdx\right)^{1/p}.
    \end{align}
    Taking a sequence $\{p_j\}$ such that $p_j \to \infty$ and applying Fatou's lemma, we obtain the result. 

    Now assume that $v$ is another strong solution of \eqref{eq:mZK}. 
Since $u,v$ are solutions to \eqref{eq:mZK}, $u - v$ solves 
\begin{align}\label{eq:diff eq}
    (\partial_t + \partial_{x_1}\Delta)(u - v) = |u|^2\partial_{x_1}u - |v|^2\partial_{x_1}v.
\end{align}
Using \eqref{eq:diff eq}, we obtain
    \begin{align*}
        \|u(t) - v(t)\|_{L^2}^2 = 2 Re\int_0^t\int (|u|^2\partial_{x_1}u - |v|^2\partial_{x_1}v)\overline{(u-v)}dxd\tau.
    \end{align*}

    Adding and subtracting $|v|^2\partial_{x_1}u(\overline{u - v})$ and rearranging the RHS, we have 
    \begin{align*}
        \|u(t) - v(t)\|_{L^2}^2 &= 2 Re\int_0^t\int (|u|^2 - |v|^2)(\overline{u - v})\partial_{x_1}u\,dxd\tau + 2Re \int_0^t\int |v|^2\partial_{x_1}(u -v)(\overline{u -v})\,dxd\tau\\
        & = 2 Re\int_0^t\int (|u|^2 - |v|^2)(\overline{u - v})\partial_{x_1}u\,dxd\tau - Re \int_0^t\int \partial_{x_1}(|v|^2)|u - v|^2\,dxd\tau.
    \end{align*}
    Using that $|u|^2 - |v|^2 = u(\overline{u - v}) + \overline{v}(u - v)$, we get 
    \begin{align}\label{eq: uniqueness diff}
    \|u(t) - v(t)\|_{L^2}^2 &= 2Re \int_0^t\int u(\overline{u - v})^2\partial_{x_1}u\,dxd\tau + 2Re \int_0^t \int \overline{v}|u - v|^2\partial_{x_1}u\,dxd\tau \nonumber\\
    &\qquad - Re \int_0^t\int \partial_{x_1}(|v|^2)|u -v|^2\,dxd\tau.
    \end{align}
    Let $M = \sup_{t \in [-T,T)}\Big\{\|u(t)\|_{H^s}, \|v(t)\|_{H^s}\Big\}$, which is finite since $u$ and $v$ are strong solutions. 

    Then \eqref{eq: uniqueness diff} becomes 
    \begin{align*}
        \sup_{0 \leq \tau \leq t}\|u(\tau) - v(\tau)\|_{L^2}^2 \leq CM\sup_{0 \leq \tau \leq t}\|u(\tau) - v(\tau)\|_{L^2}^2\int_0^t (\|\partial_{x_1}u\|_{L^\infty} + \|\partial_{x_1}v\|_{L^\infty})d\tau. 
    \end{align*}
    Since  $\partial_{x_1}u,\partial_{x_1}v \in L^1([-T,T], L^\infty)$ thanks to \eqref{eq:W1inf bound}, there exists a $t > 0$ such that we absorb the RHS into the LHS so that we get 
    \begin{align}
        \sup_{0\leq \tau \leq t}\|u(\tau) - v(\tau)\|_{L^2} = 0. 
    \end{align}
    If we define $\overline{t} = \sup\{t \in [0,T]: u(t) = v(t)\}$ we've shown that $\overline{t} > 0$. Repeating this argument starting from $\overline{t}$, we obtain $\overline{t} = T$. 
\end{proof}

\section{Failure of uniform continuity}\label{sec:fail_unif_cont}

In this section we prove the failure of uniform continuity of the flow map for \eqref{eq:mZK} by adapting the procedure outlined in \cite{Linares2019} and \cite{Herr2006} to the context of the complex-valued mZK equation \eqref{eq:mZK}. 
As in \cite{Linares2019} and \cite{Herr2006}, the proof of theorem is based on the construction of two families of approximate solutions designed to diverge as $m \to \infty$. However, these families must be tailored to the equation \eqref{eq:mZK}, taking into account that solutions of \eqref{eq:mZK} are complex-valued and that the nonlinearity is cubic.

\begin{proof}[Proof of Theorem \ref{thm:failure}]

The proof follows in 3 steps: 
\begin{enumerate}
    \item First, we construct two families of approximate solutions to the mZK equation \eqref{eq:mZK}, $\{u_{m,0}\}_{m}, \{u_{m,1}\}_{m}$ such that 
    \begin{align*}
        \|u_{m,0}(0) - u_{m,1}(0)\|_{H^s} \overset{m\to\infty}{\rightarrow} 0,
    \end{align*}
    and 
  \begin{align*}
        \lim_{m\rightarrow \infty}\inf\|u_{m,1}(t) - u_{m,0}(t)\|_{H^s} > ct \qquad \forall t > 0.
    \end{align*}
    \item Next, for $s > \frac{5}{3}$, we prove estimates on the difference between a true solution and the family of approximate solutions. That is, we show that these families are \textit{good} approximations of the true solutions. 
    \item Finally, we consider two true solutions of \eqref{eq:mZK} with $s > \frac{5}{3}$ that are well approximated by $\{u_{m,0}\}_m$ and $\{u_{m,1}\}_{m}$. We control the difference of these true solutions by the difference of the two families of approximate solutions and the estimates from step (2) to conclude. 
\end{enumerate}

\underline{\textbf{Step 1:}} We start by introducing initial data of the two families $\{u_{m,0}\}_m$ and $\{u_{m,1}\}_m$ of approximate solutions.

Let $r > 0$ and consider 
\begin{align}\label{eq:umj0}
    u_{m,j}(0,x_1,x_2): = (2\pi)^{-1}[rm^{-s}e^{i(mx_1 -x_2)} + c_{m,j}]
\end{align}
where 
\begin{align*}
c_{m,j} : =
    \begin{cases}
        m^{-1/2},& j = 1\\
        0, & j = 0. 
    \end{cases}
\end{align*}

We note that 
\begin{align}
\|u_{m,0}(0) - u_{m,1}(0)\|_{H^s} = \|m^{-1/2}\|_{H^s} = \frac{1}{\sqrt{m}} \overset{m\to\infty}{\rightarrow} 0.
\end{align}

We also observe that, by \eqref{eq:umj0}, 
\begin{align}\label{eq:L2 umj0}
    \|u_{m,j}(0)\|_{L^2(\T^2)}^2 &= (2\pi)^{-2}\int_{\T^2} r^2m^{-2s} + (c_{m,j})^2 + 2rm^{-s}c_{m,j}\cos(mx_1 - x_2)dx_1dx_2\notag\\
    & = r^2m^{-2s} + (c_{m,j})^2 + 2rm^{-s}c_{m,j}(2\pi)^{-2}\int_{\T^2}\cos(mx_1 - x_2)\,dx_1dx_2\notag\\
    & = r^2m^{-2s} + (c_{m,j})^2.
\end{align}

Now, we define the map $\mathcal{G}_t^: u_0 \in H^s(\T^2) \mapsto u_0(x_1 + t(m^2 + 1 + \|u_0\|_{L^2}^2),x_2)$.

Then, taking inspiration from \cite{Herr2006}, we define a family of approximate solutions to \eqref{eq:mZK} given by 
\begin{align*}
    u_{m,j}(t,x_1,x_2): = \mathcal{G}_t(u_{m,j}(0,x_1,x_2)), \qquad m \in \N.
\end{align*}

That is, thanks to \eqref{eq:umj0} and \eqref{eq:L2 umj0}, 
\begin{align}\label{eq:approx fam}
    u_{m,j}(t,x_1,x_2)  = (2\pi)^{-1}rm^{-s}e^{i\Phi_{m,j}(t,x_1,x_2)} + c_{m,j},
\end{align}
with 
\begin{align}
    \Phi_{m,j} := (mx_1 - x_2) + tm(m^2 + 1 + r^2m^{-2s} + c_{m,j}^2). 
\end{align}

If we now examine the difference of $u_{m,1}(t,x_1,x_2)$ and $u_{m,0}(t,x_1,x_2)$, we observe
\begin{align}
    \|u_{m,1}(t,x_1,x_2)& - u_{m,0}(t,x_1,x_2)\|_{H^s} \notag\\
    & = (2\pi)^{-1}\bigg\|rm^{-s}e^{i(mx_1 -x_2)}e^{itm(m^2+ 1 + r^2m^{-2s} + m^{-1})}+\frac{1}{\sqrt{m}} - rm^{-s}e^{i(mx_1 - x_2)}e^{itm(m^2 + 1 + r^2m^{-2s})}\bigg\|_{H^s}\notag\\
    & = (2\pi)^{-1}\bigg\|rm^{-s}e^{i(mx_1 - x_2)}e^{itm(m^2 + 1 + r^2m^{-2s})}(e^{it} - 1) + \frac{1}{\sqrt{m}}\bigg\|_{H^s}\notag\\
    & \geq \frac{r}{2}|\sin(t)| - \frac{1}{\sqrt{m}} > 0
\end{align}
for $m$ large and $t > 0$.

\underline{\textbf{Step 2:}} Fix $m > 0$. We want to control the difference between the genuine solution $u_j$ of \eqref{eq:mZK} corresponding to initial data $u_{m,j}(0)$ for $j = 0,1$, that is guaranteed by the local well-posedness theory, and the approximate solution $u_{m,j}$ so that we can exploit the specific shape of the approximate solution.

Let $v_j = u_j - u_{m,j}$. Then $v_j$ solves the following equation, 
\begin{align}
    (\partial_t + \partial_{x_1}\Delta)v_j  = |u_j|^2\partial_{x_1}u_j - |u_{m,j}|^2\partial_{x_1}u_{m,j} - R_{m,j},
\end{align}
where $R_{m,j} = (\partial_t + \partial_{x_1}\Delta)u_{m,j} - |u_{m,j}|^2\partial_{x_1}u_{m,j}$. 

We wish to obtain an energy bound for $v_j$ like we did in Lemma \ref{lem:energy bound}. Using the above equation we have 
\begin{align}
    \frac{d}{dt}\|v_j\|_{L^2}^2 &= 2Re\int_{\T^2}\dot{v}_j\overline{v}_j  \notag \\
    & = -2Re \int_{\T^2}\overline{v}_j(\partial_{x_1}\Delta v_j)  + 2Re \int_{\T^2}\overline{v}_j\Big(|u_j|^2\partial_{x_1}u_j - |u_{m,j}|^2\partial_{x_1}u_{m,j}\Big)  - 2Re\int_{\T^2} \overline{v}_jR_{m,j} \notag\\
    & = I + II + III. 
\end{align}
We will estimate each of these terms individually. We first observe that

\begin{align}
        I &= -2Re\int_{\T^2}(\partial_{x_1}\Delta v_j)\overline{v}_j  = 0.\label{eq:term I}
    \end{align}

We can decompose II using that $u_j = u_{m,j} + v_j$ so that
\begin{align}
    &|u_j|^2\partial_{x_1}u_j - |u_{m,j}|^2\partial_{x_1}u_{m,j}\notag\\
    & \qquad = |u_{m,j} + v_j|^2\partial_{x_1}(u_{m,j} + v_j) - |u_{m,j}|^2\partial_{x_1}u_{m,j}\notag\\
    & \qquad = |u_{m,j}|^2\partial_{x_1}v_j + [|v_j|^2 + 2Re(u_{m,j}\overline{v}_j)]\partial_{x_1}(u_{m,j} + v_j)\notag\\
    & \qquad = |u_{m,j}|^2\partial_{x_1}v_j + 2Re(u_{m,j}\overline{v}_j)\partial_{x_1}u_{m,j} + |v_j|^2\partial_{x_1}u_{m,j} + |v_j|^2\partial_{x_1}v_j + 2Re(u_{m,j}\overline{v}_j)\partial_{x_1}v_j.
\end{align}
We will estimate the terms in II that are quadratic and cubic in $v_j$ using H\"older's inequality and that $\|u_{m,j}\|_{L^\infty}, \|\partial_{x_1}u_{m,j}\|_{L^\infty} < C$ provided by the definition of $u_{m,j}$. We now observe that
\begin{enumerate}[label = (\roman*)]
    \item  \begin{align}\label{eq:II 1}
      2Re\int_{\T^2}2Re(u_{m,j}\overline{v}_j)\partial_{x_1}u_{m,j}\overline{v}_j &\leq C\|u_{m,j}\|_{L^\infty}\|\partial_{x_1}u_{m,j}\|_{L^\infty}\|v_j\|_{L^2}^2\notag\\
     & \leq C\|v_j\|_{L^2}^2.
 \end{align}
    \item \begin{align}\label{eq:II 2}
    2Re\int_{\T^2}|u_{m,j}|^2\partial_{x_1}v_j\overline{v}_j &= \int_{\T^2}|u_{m,j}|^2\partial_{x_1}|v_j|^2 \notag\\
     &= - \int_{\T^2}\partial_{x_1}|u_{m,j}|^2|v_j|^2\notag\\
     & \leq C\|\partial_{x_1}|u_{m,j}|^2\|_{L^\infty}\|v_j\|_{L^2}^2\notag\\
     & \leq C\|\partial_{x_1}u_{m,j}\|_{L^\infty}\|u_{m,j}\|_{L^\infty}\|v_j\|_{L^2}^2\notag\\
     & \leq C\|v_j\|_{L^2}^2.
 \end{align}
 \item \begin{align}\label{eq:II 3}
    2Re\int_{\T^2}|v_j|^2\partial_{x_1}u_{m,j}\overline{v}_j
    & \leq C\|\partial_{x_1}u_{m,j}\|_{L^\infty}\|v_j\|_{L^\infty}\|v_j\|_{L^2}^2\notag\\
    & \leq C\|v_j\|_{L^\infty}\|v_j\|_{L^2}^2.
\end{align}
\item \begin{align}\label{eq:II 4}
    4Re\int_{\T^2} Re(u_{m,j}\overline{v}_j)\partial_{x_1}v_j\overline{v}_j &\leq C\|\partial_{x_1}v_j\|_{L^\infty}\|u_{m,j}\|_{L^\infty}\|v_j\|_{L^2}^2\notag\\
    & \leq C\|\partial_{x_1}v_j\|_{L^\infty}\|v_j\|_{L^2}^2.
\end{align}
 
\end{enumerate}

Now, we are left to consider the quartic term (in $v_j$) in II.

\begin{align}\label{eq:II 5}
2Re\int_{\T^2}|v_j|^2\partial_{x_1}v_j\overline{v}_j = \frac{1}{2}\int_{\T^2}\partial_{x_1}(|v_j|^4) = 0.
\end{align}

Then, summing the estimates \eqref{eq:II 1} - \eqref{eq:II 5}, we have 
\begin{align}
    II \leq C(1 + \|v_j\|_{L^\infty} + \|\partial_{x_1}v_j\|_{L^\infty})\|v_j\|_{L^2}^2.
\end{align}

Lastly, by first applying H\"older's inequality and then Young's inequality, we have 
\begin{align}
   III =  2Re \int_{\T^2}\overline{v}_jR_{m,j} &\leq C\|v_j\|_{L^2}\|R_{m,j}\|_{L^2}\notag\\
    & \leq C\left(\|v_j\|_{L^2}^2 + \|R_{m,j}\|_{L^2}^2\right).
\end{align}

Combining these estimates on I, II, and III, we obtain 
\begin{align}\label{eq:combine estimates}
    \frac{d}{dt}\|v_j\|_{L^2}^2 &\leq C(1 + \|v_j\|_{L^\infty} + \|\partial_{x_1}v\|_{L^\infty})\|v_j\|_{L^2}^2 + C\|R_{m,j}\|_{L^2}^2
\end{align}

Since $v_j = u_j - u_{m,j}$, we have that
\begin{align}\label{eq:bound vj}
    \|v_j\|_{L^\infty} \leq \|u_j\|_{L^\infty} + \|u_{m,j}\|_{L^\infty} \leq C\|u_j\|_{L^\infty}
\end{align}
and 
\begin{align}\label{eq:bound vjx}
    \|\partial_{x_1}v_j\|_{L^\infty} \leq \|\partial_{x_1}u_j\|_{L^\infty} + \|\partial_{x_1}u_{m,j}\|_{L^\infty} \leq C\|\partial_{x_1}u_j\|_{L^\infty}.
\end{align}
We note that by Sobolev embedding, since $s > \frac{5}{3}$,  $\|u_j\|_{L^\infty} \leq \|u_j\|_{H^s}$ and thanks to the proof of Proposition \ref{prop:existence}, $\|u_j\|_{H^s} \leq C$.

Applying the estimates \eqref{eq:bound vj} and \eqref{eq:bound vjx} and integrating in time, \eqref{eq:combine estimates} becomes
\begin{align}
    \|v_j\|_{L^2}^2 \leq C\int_{0}^t\|R_{m,j}\|_{L^2}^2d\tau + C\int_0^t(1 + \|\partial_{x_1}u_j\|_{L^\infty})\|v_j\|_{L^2}^2 d\tau.
\end{align}
Then, thanks to Gronwall's inequality, 
\begin{align}\label{eq:after gronwall}
    \|v_j\|_{L^2}^2 &\leq C\int_0^t\|R_{m,j}\|_{L^2}^2\,d\tau + C\int_0^t\left(\int_0^s\|R_{m,j}\|_{L^2}^2d\tau\right)(1 + \|\partial_{x_1}u_j\|_{L^\infty})\,\text{exp}\left(\int_s^t(1 + \|\partial_{x_1}u_j\|_{L^\infty})dr\right)ds\notag\\
    & \leq C\left(\int_0^t \|R_{m,j}\|_{L^2}^2\,d\tau\right)\left(1 + \int_0^t(1 + \|\partial_{x_1}u_j\|_{L^\infty})\,\text{exp}\left(\int_s^t(1 + \|\partial_{x_1}u_j\|_{L^\infty})dr\right)ds\right)
\end{align}
for $t \in [0,T]$ with $T$ given by Lemma \ref{lem:local time}.

Before we conclude, we estimate $R_{m,j}$. Recall that 
\begin{align*}R_{m,j} = (\partial_t + \partial_{x_1}\Delta)u_{m,j} - |u_{m,j}|^2\partial_{x_1} u_{m,j}.\end{align*}

First we examine the linear part. Using the definition of $u_{m,j}$ in \eqref{eq:approx fam} we have
\begin{align*}
    \partial_tu_{m,j} &= rm^{-s}(im[m^2 + 1 + r^2m^{-2s} + c_{m,j}^2])e^{i\Phi_{m,j}(t)}
\end{align*}
and 
\begin{align*}
    (\partial_{x_1x_1x_1} + \partial_{x_1}\partial_{x_2x_2})u_{m,j} &= rm^{-s}(-im[m^2 + 1])e^{i\Phi_{m,j}(t)}\\
\end{align*}
so that 
\begin{align*}(\partial_t + \partial_{x_1}\Delta)u_{m,j} = irm^{-s}[r^2m^{-2s + 1} + mc_{m,j}^2]e^{i\Phi_{m,j}(t)} \end{align*}

Specifically, 
\begin{align}
    (\partial_t + \partial_{x_1}\Delta)u_{m,0} &= ir^3m^{1 - 3s}e^{i\Phi_{m,0}(t)} \label{eq:lin Rm0}\\
    (\partial_t + \partial_{x_1}\Delta)u_{m,1} &= [ir^3m^{1-3s} + irm^{-s}]e^{i\Phi_{m,1}(t)}\label{eq:lin Rm1}
    %[{\color{blue}ir^3m^{1-3s}} + {\color{red}irm^{-s}}]e^{i\Phi_{m,1}(t)}\label{eq:lin Rm1}
\end{align}
Next we examine the nonlinearity in $R_{m,j}$. 
\begin{align}
    |u_{m,0}|^2\partial_{x_1}u_{m,0} &= ir^3m^{1 - 3s}e^{i\Phi_{m,0}(t)}\label{eq:nonlin Rm0}\\
    |u_{m,1}|^2\partial_{x_1}u_{m,1} & = rm^{-s}\Big[r^2m^{-2s} + m^{-1} + 2rm^{-s}m^{-1/2}\cos(\Phi_{m,1}(t))\Big]ime^{i\Phi_{m,1}(t)}\notag\\
    & = \Big[ir^3m^{1-3s} + irm^{-s} + 2ir^2m^{1/2-2s}\cos(\Phi_{m,1}(t))\Big]e^{i\Phi_{m,1}(t)}.\label{eq:nonlin Rm1}
    %\Big[{\color{blue}ir^3m^{1-3s}} + {\color{red}irm^{-s}} + 2ir^2m^{1/2-2s}\cos(\Phi_{m,1}(t))\Big]e^{i\Phi_{m,1}(t)}.\label{eq:nonlin Rm1}
\end{align}

Then combining \eqref{eq:lin Rm0} with \eqref{eq:nonlin Rm0}, and \eqref{eq:lin Rm1} with \eqref{eq:nonlin Rm1}, 
\begin{align}
    R_{m,0} &= ir^3m^{1-3s}e^{i\Phi_{m,0}(t)} - ir^3m^{1-3s}e^{i\Phi_{m,0}(t)} = 0,\\
    R_{m,1} &= 2ir^2m^{1/2 - 2s}\cos(\Phi_{m,1}(t))e^{i\Phi_{m,1}(t)}.
\end{align}
Hence, 
\begin{align*}
    |R_{m,j}| = |(\partial_t + \partial_{x_1}\Delta)u_{m,j} - |u_{m,j}|^2\partial_{x_1} u_{m,j}| \leq Cm^{1/2 - 2s}.
\end{align*}

which, together with \eqref{eq:after gronwall} and \eqref{eq:L1Linf for mZK} which implies $\|\partial_{x_1}u_j\|_{L^1_TL^\infty} < C$, allows us to conclude
\begin{align}\label{eq:v L2 bound}
    \|v_j\|_{L^2} \leq Cm^{1/2 - 2s}, 
\end{align}
where $C$ depends on $r$. 

By direct calculation, using the Fourier transform,  $\|u_{m,j}(t)\|_{H^\sigma} \sim m^{\sigma - s}$. Also, by persistence of regularity and local well-posedness we obtain $\|u_{j}\|_{H^{\sigma}} \lesssim m^{\sigma - s}$. Hence, 
\begin{align}\label{eq:Hs+1}
    \|v_j\|_{H^{s+1}} \leq \|u_j\|_{H^{s+1}} + \|u_{m,j}\|_{H^{s+1}} \lesssim m.
\end{align}
Then interpolating \eqref{eq:v L2 bound} and \eqref{eq:Hs+1} we have 
\begin{align*}
    \|v_j\|_{H^s} \lesssim \|v_j\|_{L^2}^{\frac{1}{s+1}}\|v_j\|_{H^{s+1}}^{\frac{s}{s+1}} \lesssim m^{\frac{1/2 - s}{s+1}}\lesssim m^{-7/16}.
\end{align*}

\underline{\textbf{Step 3:}} Now, take $u^1_m,u^0_m$ to be the solutions of \eqref{eq:mZK} with initial conditions $u^1_m(0,x_1,x_2) = u_{m,1}(0,x_1,x_2)$ and $u^0_m(0,x_1,x_2) = u_{m,0}(0,x,y)$ as given by the local theory. Then, thanks to steps 1 and 2 we conclude
\begin{align*}
    \|u^1_m(t,\cdot,\cdot) - u^0_m(t,\cdot,\cdot)\|_{H^s} & = \|u^1_m(t,\cdot,\cdot) -u_{m,1}(t,\cdot,\cdot) + u_{m,1}(t,\cdot,\cdot) - u_{m,0}(t,\cdot,\cdot) + u_{m,0}(t,\cdot,\cdot) - u^0_m(t,\cdot,\cdot)\|_{H^s}\\
    &\geq \|u_{m,1}(t,\cdot,\cdot) - u_{m,0}(t,\cdot,\cdot)\|_{H^s} + O(m^{-7/16})\\
    & \geq \frac{r}{2}|\sin(t)| + O(m^{-1/2}) + O(m^{-7/16}) > ct
\end{align*}
for $m$ large and $t \in [0,1]$. 

\end{proof}

\appendix
\section{An auxiliary lemma}

\begin{lemma}\label{lem:sq}
If $s > 1$ then \[\|J^s(|f|^2)\|_{L^2(\T^2)} \leq \|J^s f\|_{L^2(\T^2)}^2.\]
\end{lemma}

 \begin{proof}
     We first prove the more general product estimate
\[
    \|J^s(fg)\|_{L^2(\mathbb T^2)}
    \leq
    C_s
    \|J^s f\|_{L^2(\mathbb T^2)}
    \|J^s g\|_{L^2(\mathbb T^2)}
\]
for \(s>1\). Write the Fourier series as
\[
    f(x)=\sum_{k\in \mathbb Z^2}\widehat f(k)e^{ik\cdot x},
    \qquad
    g(x)=\sum_{\ell\in \mathbb Z^2}\widehat g(\ell)e^{i\ell\cdot x}.
\]
Then
\[
    \widehat{fg}(m)
    =
    \sum_{k\in\mathbb Z^2}
    \widehat f(k)\widehat g(m-k).
\]
Therefore,
\[
    \widehat{J^s(fg)}(m)
    =
    \langle m\rangle^s
    \sum_{k\in\mathbb Z^2}
    \widehat f(k)\widehat g(m-k).
\]
Using the elementary inequality
\[
    \langle m\rangle^s
    \leq
    C_s\bigl(\langle k\rangle^s+\langle m-k\rangle^s\bigr),
\]
we get
\[
\begin{aligned}
    \left|\widehat{J^s(fg)}(m)\right|
    &\leq
    C_s
    \sum_{k\in\mathbb Z^2}
    \langle k\rangle^s
    |\widehat f(k)|
    |\widehat g(m-k)| \\
    &\quad
    +
    C_s
    \sum_{k\in\mathbb Z^2}
    |\widehat f(k)|
    \langle m-k\rangle^s
    |\widehat g(m-k)| .
\end{aligned}
\]
Equivalently,
\[
    \left|\widehat{J^s(fg)}(m)\right|
    \leq
    C_s
    \bigl(
        |\widehat{J^s f}| * |\widehat g|
    \bigr)(m)
    +
    C_s
    \bigl(
        |\widehat f| * |\widehat{J^s g}|
    \bigr)(m),
\]
where \(*\) denotes convolution on \(\mathbb Z^2\).

Taking the \(\ell^2_m(\mathbb Z^2)\) norm and applying Young's convolution
inequality gives
\[
\begin{aligned}
    \|J^s(fg)\|_{L^2(\mathbb T^2)}
    &\leq
    C_s
    \|\widehat{J^s f}\|_{\ell^2}
    \|\widehat g\|_{\ell^1}
    +
    C_s
    \|\widehat f\|_{\ell^1}
    \|\widehat{J^s g}\|_{\ell^2}.
\end{aligned}
\]
By Plancherel's theorem,
\[
    \|\widehat{J^s f}\|_{\ell^2}
    =
    \|J^s f\|_{L^2(\mathbb T^2)},
    \qquad
    \|\widehat{J^s g}\|_{\ell^2}
    =
    \|J^s g\|_{L^2(\mathbb T^2)}.
\]

It remains to control the \(\ell^1\) norms of the Fourier coefficients. Since
\(s>1\) and the dimension is \(2\), we have
\[
    \sum_{k\in\mathbb Z^2}\langle k\rangle^{-2s}<\infty.
\]
Thus, by Cauchy-Schwarz,
\[
\begin{aligned}
    \|\widehat f\|_{\ell^1}
    &=
    \sum_{k\in\mathbb Z^2}
    |\widehat f(k)| \\
    &=
    \sum_{k\in\mathbb Z^2}
    \langle k\rangle^{-s}
    \langle k\rangle^s
    |\widehat f(k)| \\
    &\leq
    \left(
        \sum_{k\in\mathbb Z^2}
        \langle k\rangle^{-2s}
    \right)^{1/2}
    \left(
        \sum_{k\in\mathbb Z^2}
        \langle k\rangle^{2s}
        |\widehat f(k)|^2
    \right)^{1/2} \\
    &\leq
    C_s
    \|J^s f\|_{L^2(\mathbb T^2)}.
\end{aligned}
\]
Similarly,
\[
    \|\widehat g\|_{\ell^1}
    \leq
    C_s
    \|J^s g\|_{L^2(\mathbb T^2)}.
\]
Substituting these two estimates into the convolution bound yields
\[
\begin{aligned}
    \|J^s(fg)\|_{L^2(\mathbb T^2)}
    &\leq
    C_s
    \|J^s f\|_{L^2}
    \|J^s g\|_{L^2}
    +
    C_s
    \|J^s f\|_{L^2}
    \|J^s g\|_{L^2} \\
    &\leq
    C_s
    \|J^s f\|_{L^2(\mathbb T^2)}
    \|J^s g\|_{L^2(\mathbb T^2)}.
\end{aligned}
\]
Now take \(g=\overline f\). Since \(J^s\) has real-valued Fourier multiplier
\(\langle k\rangle^s\), it commutes with complex conjugation in the sense that
\[
    \|J^s\overline f\|_{L^2(\mathbb T^2)}
    =
    \|J^s f\|_{L^2(\mathbb T^2)}.
\]
Therefore,
\[
\begin{aligned}
    \|J^s(|f|^2)\|_{L^2(\mathbb T^2)}
    &=
    \|J^s(f\overline f)\|_{L^2(\mathbb T^2)} \\
    &\leq
    C_s
    \|J^s f\|_{L^2(\mathbb T^2)}
    \|J^s\overline f\|_{L^2(\mathbb T^2)} \\
    &=
    C_s
    \|J^s f\|_{L^2(\mathbb T^2)}^2.
\end{aligned}
\]
This proves the claim.

 \end{proof}

 \section{Kato-Ponce Estimate on $\T^d$} \label{sec:apB}

This section is devoted to the Kato-Ponce estimate \cite{KatoPonce} on $\T^d = \R^d\backslash (2\pi \Z)^d$ expressed in terms of the multiplier $J^s$ \eqref{eq:Js on torus} for $s > \frac{d}{2}$, (see Lemma \ref{lem:KP}). As mentioned in the introduction, our approach is inspired by the proof of the Kato-Ponce estimate on $\T$ obtained by Kenig and Ionescu \cite{IoKe2007} and a transference principle formulated and proved by Roncal and Stinga \cite{RoncalStinga} for the fractional laplacian. However, since we prove the Kato-Ponce estimate for the multiplier $J^s$, we will first formulate and prove a transference formula for $J^s$.

 \subsection{Transference Formula for $J^s$}
    The point $(e^{iz_1},\dots, e^{iz_d}) \in \T^d$ is uniquely identified with $z = (z_1,\dots, z_d) \in Q_d : = (-\pi,\pi]^d$. 
 Following the example of Roncal and Stinga \cite{RoncalStinga} we define two operators. 

 \begin{definition}\label{def:rep and period}
     For a function $v$ on $\T^d$ we define its \textit{repetition} $Rv:\R^d \to \R$ by 
     \begin{align}
         (Rv)(x) = \sum\limits_{k \in \Z^d}v(x - 2\pi k)\1_{Q_d}(x - 2\pi k), \qquad x \in \R^d.
     \end{align}
     This is the $Q_d$-periodic function on $\R^d$ that coincides with $v$ on $\T^d$ (where $\T^d$ has been identified with $Q_d$). 

     For a function $u:\R^d \to \R$ we define its \textit{periodization} as the function $p_\Sigma u:\T^d \to \R$ given formally by 
     \begin{align}
         (p_\Sigma u)(z) = \sum\limits_{k \in \Z^d}u(z + 2\pi k), \qquad z \in \T^d.
     \end{align}
 \end{definition}
Now we are ready to state the Transference Formula on $\T^d$. 
 \begin{proposition}[Transference Formula] \label{prop:transference} Let $J^s_{\T^d}$ be as defined in \eqref{eq:Js on torus} and 
 \begin{align*}
     \widehat{J^s_{\R^d}g}(\xi) = (1 + |\xi|^2)^{s/2}\hat{g}(\xi).
 \end{align*}
 \begin{enumerate}[label = (\alph*)]
     \item Let $\varphi \in \mathcal{S}(\R^d)$. Then 
     \begin{align}
         [p_\Sigma(J^s_{\R^d}\varphi)](z) = J^s_{\T^d}(p_{\Sigma}\varphi)(z).
     \end{align}
     \item Let $v \in C^\infty(\T^d)$. Then its repetition $Rv$ satisfies
     \begin{align}
         \int_{\R^d}(Rv)J^s_{\R^d}\varphi\,dx = \int_{\T^d}v(z)J^s_{\T^d}(p_\Sigma \varphi)\,dz \qquad \varphi \in \mathcal{S}(\R^d).
     \end{align}
     That is, when evaluated on periodizations of Schwartz functions, the periodic distribution of $J^s_{\T^d}v$ coincides with $J^s_{\R^d}(Rv)$ in the sense of distributions. 
 \end{enumerate} 
 \end{proposition}
\begin{remark}
  In the above proposition we show a transference formula analogous to the one of Roncal and Stinga for the fractional laplacian \cite{RoncalStinga}. However, in the case of $J^s$ for $s \geq 1$ rather than the fractional laplacian, Proposition \ref{prop:transference} does not require an assumption on the decay of the Fourier coefficients of $v$. \end{remark}
 \begin{proof}

\begin{enumerate}[label = (\alph*)]
    \item
First, we justify the identity
\begin{align}\label{eq:periodic identity}
    p_{\Sigma}(J^s_{\R^d}\varphi) = J^s_{\T^d}(p_\Sigma \varphi)
\end{align}
 for $\varphi \in \mathcal{S}(\R^d)$. 

 Since $\varphi \in \mathcal{S}(\R^d)$, $J^s_{\R^d}\varphi \in \mathcal{S}(\R^d)$. Thus, for every $N > d$, we have $|J^s_{\R^d}\varphi(x)| \leq C_N(1 + |x|)^{-N}$. 

 Then 
 \begin{align}
     \left|p_{\Sigma}(J^s_{\R^d}\varphi)(x)\right| & =\left|\sum\limits_{k \in \Z^d}J^s_{\R^d}\varphi(x + 2\pi k)\right|\notag\\
     & \leq C_N\sum\limits_{k \in \Z^d}(1 + |x + 2\pi k|)^{-N}\notag \\
     &\lesssim \int_{\R^d}(1 + |x + 2\pi \xi|)^{-N}d\xi < \infty.
 \end{align}
 so the LHS of \eqref{eq:periodic identity} is well-defined.

Now we examine the RHS of \eqref{eq:periodic identity}. First, by Definition \ref{def:rep and period} of $p_{\Sigma}$,
\begin{align}
    p_{\Sigma}\varphi(z) &= \sum\limits_{k \in \Z^d}\varphi(z + 2\pi k)\notag\\
    & = \sum\limits_{k \in \Z^d}\left(\int_{\R^d}\widehat{\varphi}(\xi)e^{i\xi\cdot(z + 2\pi k)}d\xi\right).\notag
\end{align}
Let $g_z(\xi) = \widehat{\varphi}(\xi)e^{i\xi\cdot z}$. Then, applying the Poisson summation formula \cite[Theorem 3.2.8]{Grafakos},
\begin{align}\label{eq:get fourier coefficients}
    p_{\Sigma}\varphi(z) & = \sum\limits_{k \in \Z^d}(g_z)^{\vee}(k) = \sum\limits_{m \in \Z^d}g_z(m) = \sum\limits_{m \in \Z^d}\widehat{\varphi}(m)e^{im\cdot z},
\end{align}
which implies \begin{align}\label{eq:ps fourier coeff}\widehat{(p_\Sigma\varphi)}(m) = \widehat{\varphi}(m)\end{align} for each $m \in \Z^d$. 

Using \eqref{eq:get fourier coefficients}, we can compute the periodization of $J^s_{\R^d}\varphi$ as follows,
\begin{align}
    [p_\Sigma(J^s_{\R^d}\varphi)](z)& = \sum\limits_{k \in \Z^d}\widehat{J^s_{\R^d}\varphi}(k)e^{ik\cdot z} \notag\\
    & = \sum\limits_{k \in \Z^d}\langle k \rangle^{s}\widehat{\varphi}(k)e^{ik\cdot z}\notag\\
    & = \sum\limits_{k \in \Z^d}\langle k \rangle^s \widehat{p_{\Sigma}\varphi}(k)e^{ik\cdot z}\notag\\
    & = \sum\limits_{k \in \Z^d}\widehat{J^s_{\T^d}p_\Sigma \varphi}(k)e^{ik\cdot z}\notag\\
    & = J^s_{\T^d}(p_\Sigma \varphi)(z)
\end{align}
for each $z \in \T^d$. 

\item 
Then applying the result (a), we obtain 
\begin{align}
    \int_{\R^d}(Rv)J^s_{\R^d}\varphi\,dx &= \int_{\R^d}\Big[\sum\limits_{k \in \Z^d}v(x - 2\pi k)\1_{Q_d}(x - 2\pi k)\Big]J^s_{\R^d}\varphi(x)\,dx \notag\\
    & = \sum\limits_{k \in \Z^d}\int_{Q_d + 2\pi k}v(x - 2\pi k)J^s_{\R^d}\varphi(x)\,dx \notag\\
    & = \sum\limits_{k \in \Z^d}\int_{Q_d}v(z)J^s_{\R^d}\varphi(z + 2\pi k)\,dz\notag\\
    & = \int_{\T^d}v(z)[p_\Sigma(J^s_{\R^d}\varphi)](z)\,dz \notag\\
    & = \int_{\T^d}v(z)J^s_{\T^d}(p_\Sigma \varphi)\,dz,
\end{align}
where $Q_d = (-\pi,\pi]^d$.

\end{enumerate}
 \end{proof}

\begin{lemma}\label{lem:periodization}
     Let $h \in \mathcal{S}(\R^d)$. 
     \begin{align}
         \|p_\Sigma h\|_{L^2(\T^d)} \lesssim \|h\|_{L^2(\R^d)}
     \end{align}
 \end{lemma}

 \begin{proof}
     Recall $Q_d = [-\pi,\pi)^d$. Then for $x \in Q_d$, \begin{align*}
         p_\Sigma h(x) = \sum\limits_{k \in \Z^{d}}h(x+2\pi k)
     \end{align*}

     Integrating over $Q_d$, 
     \begin{align*}
         \|p_\Sigma h\|_{L^2(\T^d)}^2 &= \int_{Q_d}|p_\Sigma h(x)|^2\,dx\\
         & \leq\sum\limits_{k \in \Z^d}\int_{Q_d}|h(x + 2\pi k)|^2\,dx.
     \end{align*}
     Now let $y = x + 2\pi k$. Then $y \in Q_d + 2\pi k$ and $Q_d + 2\pi k$ tile $\R^d$. Thus, 
     \begin{align*}
         \|p_\Sigma h\|_{L^2(\T^d)}^2 &\leq \sum\limits_{k \in \Z^d}\int_{Q_d + 2\pi k}|h(y)|^2\,dy\\
         & \leq \|h\|_{L^2(\R^d)}
     \end{align*}
\end{proof}
\subsection{Statement and proof of Kato-Ponce estimate on $\T^d$}
 \begin{lemma}[Kato-Ponce Inequality for $\T^d$]\label{lem:KP}
    Let $s \geq 1$ and $f,g \in H^s(\T^d)$. Then 
    \begin{align}\label{eq:KP for period}
        \|J^s_{\T^d}(fg) - fJ_{\T^d}^sg\|_{L^2(\T^d)} \leq c\Big\{\|J^s_{\T^d}f\|_{L^2(\T^d)}\|g\|_{L^\infty(\T^d)} + (\|f\|_{L^\infty(\T^d)} + \|\nabla f\|_{L^\infty(\T^d)})\|J^{s-1}_{\T^d}g\|_{L^2(\T^d)}\Big\}.
    \end{align}
\end{lemma}

\begin{proof}
It suffices to prove the estimate for \(f,g\in C^\infty(\mathbb T^d)\), since the general case follows by density. Choose finitely many cutoffs \(\eta_i\in C_c^\infty(\mathbb R^d)\), \(1\leq i\leq N\), such that
\[
\sum_{i=1}^N p_\Sigma \eta_i=1,
\qquad\text{on }\mathbb T^d
\]
and $\eta_i$ is supported on an open cube $Q_i \subset \R^d$ with side length less than $2\pi$, which guarantees the quotient map $\pi: \R^d \to \T^d$ is injective.

Let $Rf:
\R^d \to \R$ and $Rg: \R^d \to \R$ denote the periodic extensions of $f:\T^d \to \R$ and $g:\T^d \to \R$ given by Definition \ref{def:rep and period}. In particular, we note that  
\begin{align*}
    Rf(x + 2\pi k) = f(x), \qquad Rg(x + 2\pi k) = g(x),\qquad \forall k \in \Z^d.
\end{align*}

Set
\begin{align}\label{eq:def gi}
g_i=\eta_i(Rg)
\end{align}
Then \(g_i\in C_c^\infty(\mathbb R^d)\) and \(\sum_i p_\Sigma g_i=g\). 

By definition of $Rf$ and $p_\Sigma$ we have,
\begin{align}\label{eq:move f 1}
f(x)\,p_\Sigma\bigl(g_i\bigr)(x) &= f(x) \sum\limits_{k \in \Z^d}g_i(x + 2\pi k)\notag\\
&= \sum\limits_{k \in \Z^d}(Rf)(x + 2\pi k)g_i(x + 2\pi k)\notag\\
& = p_\Sigma\bigl((Rf)g_i\bigr).
\end{align}

As a consequence of \eqref{eq:move f 1}, we have the the following useful identity,
\begin{align}\label{eq:move f 2}
fg = f\sum\limits_{i=1}^N p_{\Sigma}g_i = \sum\limits_{i=1}^Np_{\Sigma}((Rf)g_i).
\end{align}

Using the transference formula for \(J^s\) given in Proposition \ref{prop:transference} (a) and \eqref{eq:move f 2}, we obtain

\[
J^s_{\mathbb T^d}(fg) = 
 \sum\limits_{i=1}^N J^s_{\T^d}p_\Sigma(\widetilde{f}g_i) =
\sum_{i=1}^N
p_\Sigma\bigl(J^s_{\mathbb R^d}((Rf)g_i)\bigr),
\]
and similarly
\[
fJ^s_{\mathbb T^d}g
=
\sum_{i=1}^N
p_\Sigma\bigl((Rf)J^s_{\mathbb R^d}g_i\bigr).
\]

Therefore, 
\begin{align}\label{eq:equality}
J^s_{\mathbb T^d}(fg)-fJ^s_{\mathbb T^d}g
=
\sum_{i=1}^N
p_\Sigma\left(
J^s_{\mathbb R^d}((Rf) g_i)
-
(Rf)J^s_{\mathbb R^d}g_i
\right).
\end{align}

Choose \(\theta_i\in C_c^\infty(\mathbb R^d)\) such that \(\theta_i\equiv1\) on a neighborhood of \(\operatorname{supp}\eta_i\). Thanks to \eqref{eq:def gi}, we have \((Rf) g_i=(\theta_i(Rf))g_i\). Hence
\begin{align}\label{eq:split }
J^s_{\mathbb R^d}((Rf) g_i)
-
\widetilde fJ^s_{\mathbb R^d}g_i
=
[J^s_{\mathbb R^d},\theta_i(Rf)]g_i
+
(\theta_i-1)(Rf)J^s_{\mathbb R^d}g_i.
\end{align}

Applying the triangle inequality and Lemma \ref{lem:periodization},
\begin{align}\label{eq:two terms}
&\|p_\Sigma([J^s_{\R^d},\theta_i(Rf)]g_i + (\theta_i - 1)(Rf)J^s_{\R^d}g_i)\|_{L^2(\T^d)} \notag\\
&\qquad \qquad \leq \left\|p_\Sigma\left([J^s_{\R^d},\theta_i(Rf)]g_i\right)\right\|_{L^2(\T^d)} + \left\|p_\Sigma\Big((1 - \theta_i)(Rf)J^s_{\R^d}g_i\Big)\right\|_{L^2(\T^d)}\notag\\
& \qquad \qquad \leq \|[J^s_{\R^d},\theta_i(Rf)]g_i\|_{L^2(\R^d)} + \left\|p_\Sigma\Big((1 - \theta_i)(Rf)J^s_{\R^d}g_i\Big)\right\|_{L^2(\R^d)}
\end{align}
The first term is controlled by the Euclidean Kato-Ponce estimate \cite{KatoPonce}:

\begin{align}\label{eq:Euc KP with cutoffs}
\|[J^s_{\R^d},\theta_i(Rf)]g_i\|_{L^2(\R^d)} &= \|J^s_{\mathbb R^d}(\theta_i(Rf)g_i) - \theta_i(Rf)J^s_{\R^d}g_i\|_{L^2(\mathbb R^d)}\notag\\
&\lesssim
\|J^s_{\mathbb R^d}(\theta_i(Rf))\|_{L^2(\mathbb R^d)}
\|g_i\|_{L^\infty(\mathbb R^d)} \notag \\
&\quad+
\left(
\|\theta_i(Rf)\|_{L^\infty(\mathbb R^d)}
+
\|\nabla(\theta_i(Rf))\|_{L^\infty(\mathbb R^d)}
\right)
\|J^{s-1}_{\mathbb R^d}g_i\|_{L^2(\mathbb R^d)}.
\end{align}

Since multiplication by a fixed compactly supported cutoff maps periodic Sobolev spaces into Euclidean Sobolev spaces,
\[
\|J^s_{\mathbb R^d}(\theta_i(Rf))\|_{L^2(\mathbb R^d)}
\lesssim
\|J^s_{\mathbb T^d}f\|_{L^2(\mathbb T^d)}
\]
and
\[
\|J^{s-1}_{\mathbb R^d}g_i\|_{L^2(\mathbb R^d)}
\lesssim
\|J^{s-1}_{\mathbb T^d}g\|_{L^2(\mathbb T^d)}.
\]
Moreover,
\[
\|g_i\|_{L^\infty(\mathbb R^d)}
\lesssim
\|g\|_{L^\infty(\mathbb T^d)}
\]
and
\[
\|\theta_i(Rf)\|_{L^\infty}
+
\|\nabla(\theta_i(Rf))\|_{L^\infty}
\lesssim
\|f\|_{L^\infty(\mathbb T^d)}
+
\|\nabla f\|_{L^\infty(\mathbb T^d)}.
\]
Thus, returning to \eqref{eq:Euc KP with cutoffs}

\begin{align}\label{eq:part of KP}
\|[J^s_{\mathbb R^d},\theta_i(Rf)]g_i\|_{L^2(\mathbb R^d)}
&\lesssim
\|J^s_{\mathbb T^d}f\|_{L^2(\mathbb T^d)}
\|g\|_{L^\infty(\mathbb T^d)} \notag\\
&\quad+
\left(
\|f\|_{L^\infty(\mathbb T^d)}
+
\|\nabla f\|_{L^\infty(\mathbb T^d)}
\right)
\|J^{s-1}_{\mathbb T^d}g\|_{L^2(\mathbb T^d)}.
\end{align}

Hence, in order to control \eqref{eq:two terms}, it remains to estimate the error
\[
(\theta_i-1)(Rf)J^s_{\mathbb R^d}g_i.
\]
Since \(\theta_i\equiv1\) on a neighborhood of \(\operatorname{supp}g_i\), this is an off-support term. The convolution kernel of \(J^s_{\mathbb R^d}\) is smooth away from the origin, and therefore
we have 
\begin{align*}
    |(1 - \theta_i)(x)J^s_{\R^d}g_i(x)| \lesssim C_s \|g_i\|_{L^\infty}\langle x \rangle^{-d-s}.
\end{align*}
Thus, since $Rf$ is periodic, 
\begin{align}\label{eq:error estimate}
\left\|p_{\Sigma}\Big((1 - \theta_i)(Rf)J^s_{\mathbb R^d}g_i\Big)\right\|_{L^2(\T^d)} &= \|fp_\Sigma\Big((1 - \theta_i)J^s_{\R^d}g_i\Big)\|_{L^2(\T^d)}\notag\\
& \leq \|f\|_{L^2(\T^d)}\|p_\Sigma\Big((1 - \theta_i)J^s_{\R^d} g_i)\|_{L^\infty(\T^d)}
\end{align}

Using the definition of $p_\Sigma$ and the pointwise bound above, we have 
\begin{align}
    \|p_\Sigma\Big((1 - \theta_i)J^s_{\R^d} g_i)\|_{L^\infty(\T^d)} &\leq  \sup_{x \in \T^d}C_s\|g_i\|_{L^\infty(\R^d)}\sum\limits_{k \in \Z^d}\langle x + 2\pi k\rangle^{-d-s}\notag\\
    & \leq C_s \|g_i\|_{L^\infty}. 
\end{align}

Thus, since $\|g_i\|_{L^\infty(\R^d)} \lesssim \|g\|_{L^\infty(\T^d)}$, the error term is dominated by $C \|f\|_{L^2(\T^d)}\|g\|_{L^\infty(\T^d)}$. 

Finally, this estimate combined with \eqref{eq:equality} - \eqref{eq:two terms} and \eqref{eq:part of KP}, summing in $1 \leq i\leq N$, we obtain
\begin{align}
    \|J^s_{\mathbb T^d}(fg)-fJ^s_{\mathbb T^d}g\|_{L^2(\T^d)} &\lesssim \Bigg(\|J^s_{\mathbb T^d}f\|_{L^2(\mathbb T^d)}
\|g\|_{L^\infty(\mathbb T^d)} \notag\\
&\quad+
\left(
\|f\|_{L^\infty(\mathbb T^d)}
+
\|\nabla f\|_{L^\infty(\mathbb T^d)}
\right)
\|J^{s-1}_{\mathbb T^d}g\|_{L^2(\mathbb T^d)}\Bigg).
\end{align}

\end{proof}

\end{document}